\documentclass[reqno, letterpaper, oneside]{article}
\usepackage{amsmath,amsfonts,amssymb,amsthm}

\usepackage{bbm}
\usepackage{setspace}
\usepackage{geometry}
\usepackage{enumerate}

\usepackage[colorlinks=false,linktoc=all,bookmarksopen=true,linktocpage,pdftex]{hyperref}
\usepackage[pdftex]{bookmark}
\hypersetup{
	colorlinks=true,       
	linkcolor=blue,        
	citecolor=blue,        
	filecolor=magenta,     
	urlcolor=blue         
}

\usepackage{tabularx} 

\usepackage{mathtools}
\usepackage{cleveref}
\usepackage{enumitem}

\theoremstyle{plain}
\newtheorem{theorem}{Theorem}[section]
\newtheorem{lemma}[theorem]{Lemma}

\newtheorem{corollary}[theorem]{Corollary}

\newtheorem{remark}[theorem]{Remark}

\theoremstyle{definition}

\newtheorem{problem}[theorem]{Problem}

\newcommand{\refT}[1]{Theorem~\ref{#1}}

\newcommand{\refL}[1]{Lemma~\ref{#1}}
\newcommand{\refR}[1]{Remark~\ref{#1}}
\newcommand{\refS}[1]{Section~\ref{#1}}

\newcommand{\cC}{\mathcal{C}}
\newcommand{\cD}{\mathcal{D}}

\newcommand{\cU}{\mathcal{U}}
\newcommand{\cN}{\mathcal{N}}

\newcommand{\acc}{\mathcal{C}^+}
\newcommand{\cc}{\mathcal{C}}
\newcommand{\portion}{\nu}

\DeclareMathOperator{\Bin}{Bin}

\newcommand{\Gnp}{G_{n,p}}

\newcommand{\indic}[1]{\mathbbm{1}_{\{{#1}\}}}

\newcommand\bigpar[1]{\bigl(#1\bigr)}
\newcommand\Bigpar[1]{\Bigl(#1\Bigr)}
\newcommand\biggpar[1]{\biggl(#1\biggr)}
\newcommand\Biggpar[1]{\Biggl(#1\Biggr)}
\newcommand\lrpar[1]{\left(#1\right)}

\newcommand\Bigsqpar[1]{\Bigl[#1\Bigr]}
\newcommand\biggsqpar[1]{\biggl[#1\biggr]}
\newcommand\Biggsqpar[1]{\Biggl[#1\Biggr]}

\newcommand\bigcpar[1]{\bigl\{#1\bigr\}}
\newcommand\Bigcpar[1]{\Bigl\{#1\Bigr\}}
\newcommand\biggcpar[1]{\biggl\{#1\biggr\}}
\newcommand\Biggcpar[1]{\Biggl\{#1\Biggr\}}

\newcommand\ceil[1]{\lceil#1\rceil}

\newcommand\biggceil[1]{\biggl\lceil#1\biggr\rceil}

\newcommand\floor[1]{\lfloor#1\rfloor}

\newcommand\Bigfloor[1]{\Bigl\lfloor#1\Bigr\rfloor}
\newcommand\biggfloor[1]{\biggl\lfloor#1\biggr\rfloor}

\newcommand{\gr}{\zeta}

\renewcommand{\Pr}{\mathbb{P}}
\newcommand{\E}{\mathbb{E}}

\newcommand{\eps}{\varepsilon}

\newenvironment{romenumerate}[1][-5pt]{
\addtolength{\leftmargini}{#1}\begin{enumerate}
 }{\end{enumerate}}

\let\OLDthebibliography\thebibliography
\renewcommand\thebibliography[1]{
  \OLDthebibliography{#1}
  \setlength{\parskip}{0pt}
  \setlength{\itemsep}{0pt plus 0.3ex}
}

\title{The clique chromatic number of sparse random graphs}
\author{Manuel Fernandez V%
\thanks{School of Mathematics, Georgia Institute of Technology, Atlanta GA~30332, USA. E-mail: {\tt mfernandez39@gatech.edu}.}
\ and Lutz Warnke%
	\thanks{Department of Mathematics, University of California, San Diego, La Jolla CA~92093, USA. 
E-mail: {\tt lwarnke@ucsd.edu}. 
Supported by NSF~CAREER grant~DMS-2225631, and a Sloan Research Fellowship.}
}
\date{March~11, 2024; revised March~22, 2025}

\begin{document}

\maketitle 

\begin{abstract}
The clique chromatic number of a graph is the smallest number of colors in a vertex coloring so that no maximal clique is monochromatic. 
In this paper, we determine the order of magnitude of the clique chromatic number of the random graph~$G_{n,p}$ for most edge-probabilities~$p$ in the range ${n^{-2/5} \ll p \ll 1}$. 
This resolves open problems and questions of Lichev, Mitsche and Warnke as well as Alon and~Krievelevich. 

One major proof difficulty stems from high-degree vertices, which prevent maximal cliques in their neighborhoods: 
we deal with these vertices by an intricate union bound argument, that combines the probabilistic method with new degree counting arguments in order to enable Janson's inequality. 
This way we  determine the asymptotics of the clique chromatic number of~$G_{n,p}$ in some ranges, 
and discover a surprising new phenomenon that contradicts earlier predictions for edge-probabilities~$p$ close~to~$n^{-2/5}$. 
\end{abstract}

\section{Introduction}\label{sec:intro}
What can we say about the typical behavior of the chromatic number~$\chi(\Gnp)$ of a random graph?
This question has tremendously impacted the development of probabilistic combinatorics~\mbox{\cite{SS1987,BB1988,AK1997,AN2005,H2021}}. 
We revisit this fundamental question for a widely-studied variant of the chromatic number  
that is more difficult to analyze, 
due to non-standard properties such as (i)~lack of monotonicity 
and (ii)~multiple different typical behaviors. 
In particular, despite considerable attention~\cite{MMP2019,AK2018,DZ2023,LMW2023}, 
even the typical order of magnitude of the clique chromatic number~$\chi_c(\Gnp)$ of the binomial random graph~$\Gnp$ 
remained an open problem for edge-probabilities~$p$ in the range~${n^{-2/5} \ll p \ll 1}$.
In this paper we essentially close this knowledge gap regarding~$\chi_c(\Gnp)$, 
and discover a surprising new behavior of~$\chi_c(\Gnp)$ for~$p=n^{-2/5+o(1)}$ that contradicts earlier~predictions. 


The \emph{clique chromatic number} $\chi_c(G)$ of a graph~$G$ 
is the smallest number of colors required to color the vertices of~$G$ so that no inclusion-maximal clique is monochromatic (ignoring isolated vertices). 
There is no simple connection between~$\chi_c(G)$ and the normal chromatic number~$\chi(G)$, which by definition satisfies~$\chi_c(G) \le \chi(G)$.
For example, we have equality~$\chi_c(G) = \chi(G)$ for triangle-free graphs~$G$, and the large gap ${2 = \chi_c(K_n)} \ll {\chi(K_n)=n}$ for complete graphs.
Two further differences are that $\chi_c(G)$ is not monotone with respect to taking subgraphs, and that the algorithmic problem of deciding whether~$\chi_c(G) \le 2$ is already NP-complete~\cite{KT2002,BGGPS2004,Marx2011}; for more structural results on~$\chi_c(G)$ see~\cite{DSSW1991,MS1999,SLK2014,CPTT2016,MMP2018,FPS2018,GPRS2021}. 
Overall, these non-standard properties suggest that the clique chromatic number~$\chi_c(G)$ is harder to determine than the usual chromatic~number~$\chi(G)$.

McDiarmid, Mitsche and Pra{\l}at~\cite{MMP2019} initiated the study of the clique chromatic number~$\chi_c(\Gnp)$ of the 
binomial random graph~$\Gnp$ with edge-probability~$p=p(n)$ and vertex set~$V=V(\Gnp):=\{1, \ldots, n\}$.
Perhaps surprisingly, they discovered three different behaviors of~$\chi_c(\Gnp)$:  for any fixed~$\eps>0$, they showed that typically  
\begin{equation}\label{eq:previous}
 \chi_c\bigpar{\Gnp} = 
\begin{cases}
\Theta\Bigpar{\frac{np}{\log(np)}} \quad   & \: \text{if $n^{-1} \ll p \le n^{-1/2-\eps}$,} \\
\Theta\Bigpar{\frac{p^{3/2}n}{\sqrt{\log n}}} & \: \text{if $n^{-1/2+\eps} \le p \le n^{-2/5-\eps}$,} \\
\tilde{\Theta}\Bigpar{\frac{1}{p}} & \: \text{if $n^{-2/5+\eps} \le p \le n^{-1/3-\eps}$ or $n^{-1/3+\eps} \le p \le n^{-\eps}$,} 
\end{cases}
\end{equation}
where the $\tilde{\Theta}$-notation suppresses extra logarithmic factors, as usual. 
In concrete words, \eqref{eq:previous} determines the typical value of~$\chi_c(\Gnp)$ 
up to constants factors for most~$p$ in the range~${n^{-1} \ll p \le n^{-2/5-\eps}}$, and up to logarithmic factors for most~$p$ in the range~${n^{-2/5+\eps} \le p \le n^{-\eps}}$. 
The fact that the form of~$\chi_c(\Gnp)$ changes in different edge-probability ranges adds to the increased level of difficulty, 
which intuitively explains why in~\eqref{eq:previous} we only know the `correct' order of magnitude of~$ \chi_c(\Gnp)$ in two out of three cases. 
The behavior of~$\chi_c(\Gnp)$ in the intermediate range~$n^{-1/2-\eps} \le p \le n^{-1/2+\eps}$ was clarified by Lichev, Mitsche and Warnke~\cite{LMW2023}, 
but their approach did not fully extend to~${p \ge n^{-2/5+\eps}}$, which they left as an open problem.
The typical order of~$\chi_c(\Gnp)$ for constant~$p \in (0,1)$ was resolved by Alon and Krivelevich~\cite{AK2018}, 
who pointed out that their argument does not give the `correct' order for~${p \ll 1}$, 
i.e., they raised the question what happens in the sparse case; see also~\cite{DZ2023}. 
To sum up: the logarithmic gaps in~\eqref{eq:previous} 
for~$n^{2/5+\eps} \le p \ll 1$  
remained the main open problem for~$\chi_c(\Gnp)$. 

\subsection{Main results}
The main result of this paper determines the typical order of magnitude of the clique chromatic number~$\chi_c(\Gnp)$ 
for most edge-probabilities~$p$ in the range~$n^{2/5+\eps} \le p \ll 1$. 
This closes the logarithmic gaps that were present until now, and 
resolves the aforementioned open problems and questions of~Lichev, Mitsche and Warnke~\cite{LMW2023} and Alon and Krivelevich~\cite{AK2018}, respectively. 
%
\begin{theorem}[Main result: order of magnitude]\label{thm:main1}%
Fix~$\eps >0$. 
If the edge-probability~$p=p(n)$ satisfies ${n^{-2/5 + \eps} \ll p \ll n^{-1/3}}$ or ${n^{-1/3 + \eps} \ll p \ll 1}$,
then with high~probability (i.e., with probability tending to one as $n \to \infty$) 
the clique chromatic number of the \mbox{binomial} random graph~$\Gnp$ satisfies
\begin{equation}\label{eq:main1}
    \chi_c\bigpar{\Gnp} \: = \: \Theta\bigpar{\log(n)/p} .
\end{equation}
\end{theorem}
The main contribution of \refT{thm:main1} is the lower bound in~\eqref{eq:main1}.
Previous work~\cite{MMP2019,LMW2023} obtained the bound ${\chi_c(\Gnp) = \Omega(1/p)}$ by showing that no color class can have `too many' vertices, 
and our logarithmic improvement required us to look at all color classes simultaneously (instead of one color class at a time).  
Indeed, our proof strategy for the improved lower bound ${\chi_c(\Gnp) = \Omega(\log(n)/p)}$ is to show that, 
in any vertex coloring of~$\Gnp$ with `too few' colors, there is an inclusion-maximal clique in some color class. 
To prove the existence of these cliques, we shall use an intricate union bound argument that
relies on the probabilistic method and new degree counting arguments,  
in order to work around the high-degree vertices bottleneck from previous work~\cite{MMP2019,LMW2023,AK2018}. 
In particular, this argument enables us to construct a sufficiently large collection of inclusion-maximal cliques candidates 
in one color class that is amenable to an application of Janson's inequality (a large-deviation inequality for non-existence),
which in turn allows us to show that at least one of these candidates is in fact an inclusion-maximal clique in~$\Gnp$; 
see \refS{sec:strategy} for more~details about our proof~strategy.


To improve our understanding of the clique chromatic number~$\chi_c(\Gnp)$ of~$\Gnp$, 
in view of \refT{thm:main1} it is desirable to understand 
(i)~if we can determine asymptotics of $\chi_c(\Gnp)$ for some~$p$ in the range~$n^{-2/5+\eps} \le p \ll 1$, and 
(ii)~if the extra~$n^{\eps}$ factors in the assumptions $p \gg n^{-2/5 + \eps}$ and~$p \gg n^{-1/3 + \eps}$ of \refT{thm:main1} are necessary. 
Our proof techniques are powerful enough to provide some insights into these natural questions. 
In particular, \refT{thm:main2} determines the typical asymptotics of~$\chi_c(\Gnp)$ in the sparse range~${n^{-o(1)} \le p \ll 1}$. 
\begin{theorem}[Sparse asymptotics]\label{thm:main2}%
\sloppy 
If the edge-probability $p=p(n)$ satisfies ${n^{-o(1)} \le p \ll 1}$,
then with high~probability the clique chromatic number of the binomial random graph~$\Gnp$ satisfies
\begin{equation}\label{eq:main2}
    \chi_c\bigpar{\Gnp} \: = \:     \bigpar{\tfrac{1}{2}+o(1)}  \log(n)/p .
\end{equation}
\end{theorem}
The typical asymptotics~\eqref{eq:main2} of $\chi_c(\Gnp)$ are consistent with a result of Demidovich and Zhukovskii~\cite{DZ2023} for constant~${p \in [0.5,1)}$, 
and thus it is natural to wonder whether \refT{thm:main2} can be extended to smaller edge-probabilities~$p=p(n)$. 
It turns out that this is not possible: the leading constant in~\eqref{eq:main2} must be smaller than~$1/2$ for~$n^{-2/5} \le p \le n^{-\Omega(1)}$ due to the upper bound of \refL{lem:upper-bound} in \refS{sec:upper}. 
This difference is further highlighted by \refT{thm:main3}, which determines the typical asymptotics of~$\chi_c(\Gnp)$ for very small edge-probabilities~$p=p(n)$. 
\begin{theorem}[Very sparse asymptotics]\label{thm:main3}%
\sloppy 
Suppose that~$\omega=\omega(n) \gg 1$. If the edge-probability~$p=p(n)$ satisfies ${(\log n)^{\omega}n^{-2/5} \leq p \ll n^{-1/3}}$,
then with high~probability the clique chromatic number of the binomial random graph~$\Gnp$ satisfies
\begin{equation}\label{eq:main3}
    \chi_c\bigpar{\Gnp} \: = \:     \bigpar{\tfrac{5}{2} +o(1) }\log\bigpar{n^{2/5}p}/p .
\end{equation}
\end{theorem}
%
Writing~$p=n^{-2/5+\eps}$, note that \refT{thm:main3} implies that $\chi_c(\Gnp) =(1+o(1)) 5 \eps/2 \cdot \log(n)/p$ holds with high~probability, 
which in concrete words says that the leading constant in front of~$\log(n)/p$ changes `adaptively' with~$p$.
This is a new and surprising phenomenon, 
which in particular implies (formally using \refL{lem:upper-bound-epsilon-1} from \refS{sec:upper} to slightly decrease the assumed lower bound on~$p$) 
 that with high~probability 
\begin{equation}\label{eq:main3Lcor}
   \chi_c\bigpar{\Gnp} \: = \:  o\bigpar{\log(n)/p} \qquad  \:  \text{if~$n^{-2/5} \ll p \le n^{-2/5+o(1)}$.}
\end{equation}
This `smallness' conclusion is interesting for several reasons.
Namely, \eqref{eq:main3Lcor} shows that in \refT{thm:main1} the assumption~${p \gg n^{-2/5 + \eps}}$ with fixed~$\eps>0$ is necessary to conclude~\eqref{eq:main1}. 
Furthermore, \eqref{eq:main3Lcor} contradicts the earlier prediction of Lichev, Mitsche and~Warnke~\cite{LMW2023} that typically $\chi_c(\Gnp) = \Theta(\log(n)/p)$ 
 for~${n^{-2/5}  (\log n)^{3/5} \ll p \ll 1}$. 
It may be possible that the typical behavior of $\chi_c(\Gnp)$ undergoes additional surprises when~$p$ is around~$n^{-1/3}$, 
which is one of several interesting open problems that we discuss in the concluding \refS{sec:concluding} of this paper.

\subsection{Proof difficulties and strategy}\label{sec:strategy}
In the following we informally discuss some difficulties that our proof of \refT{thm:main1} needs to overcome, 
in order to establish the lower bound~$\chi_c(\Gnp) = \Omega(\log(n)/p)$ on the clique chromatic number of the random graph~$\Gnp$. 

The natural proof approach would be a union bound argument over all color classes, which unfortunately does not give the desired lower bound. 
Indeed, if we try to prove ${\chi_c(\Gnp) > n/r}$ by showing that every set of~$r$ vertices contains a clique that is inclusion-maximal in~$\Gnp$, then we get the wrong logarithmic factor. 
To see this, consider a set~$U$ in the neighborhood of a vertex: by construction~$U$ contains no inclusion-maximal cliques, which together with the typical neighborhood size implies that the aforementioned argument requires~${r = \Omega(np)}$. 
It follows that this kind of union bound argument can at best give ${\chi_c(\Gnp) = \Omega(1/p)}$, which in fact was established for many edge-probabilities~$p=p(n)$ by McDiarmid, Mitsche and Pra{\l}at~\cite{MMP2019}; see also~\cite{LMW2023}.

\enlargethispage{\baselineskip}

The basic idea for bypassing the discussed union bound bottleneck is to look at all color classes simultaneously. 
More concretely, for suitable ${s=\Theta(\log(n)/p)}$ the plan is to argue that for any vertex partition into $s$ color classes, there is at least one color class which contains a clique that is inclusion-maximal in~$\Gnp$, which in turn implies that $\chi_c(\Gnp) > s={\Theta(\log(n)/p)}$. 
To implement this plan, we would like to use the following three-step approach:%
{\vspace{-0.25em}
\begin{itemize} 
\itemsep 0em \parskip 0.125em  \partopsep=0.125em \parsep 0.125em  
\item[(i)] first we show that any such \mbox{$s$ vertex coloring} contains at least one `useful' color class $W$, i.e., where every vertex outside of~$W$ has a reasonable number of non-neighbors in~$W$, 
\item[(ii)] then we show that any such useful~$W$ contains many maximal clique candidates, i.e., subsets~$K \subseteq W$ which have no common neighbor outside of~$W$, and 
\item[(iii)] finally we show that at least one of these maximal clique candidates~$K$ is a clique in~$\Gnp$, by combining Janson's inequality (a large-deviation inequality for non-existence) with a union bound argument over~all~$W$.\vspace{-0.25em} 
\end{itemize}\vspace{-0.25em}}%
\noindent 
For constant edge-probabilities~$p \in (0,1)$ this plan was implemented by Alon and Krivelevich~\cite{AK2018}, 
who pointed out that their argument gives $\chi_c(\Gnp) = {\Omega(\log(n)/\log(1/p))}$ in the sparse case~${p \to 0}$,
which is weaker than our desired lower bound. 
In the sparse case, 
akin to the analysis of large independent sets in random constructions of clique-free graphs~\cite{KK3,BK3,BKH,LWK4,GWK3}, 
the main bottleneck for the three-step proof approach are vertices outside of~$W$ with many neighbors in~$W$:  
these can severely limit the number of maximal clique candidates inside~$W$ in step~(ii), 
which in turn prevents Janson's inequality from yielding good enough failure probabilities in~step~(iii). 

The core idea for limiting the impact of the high-degree vertices outside of~$W$ 
is to combine the probabilistic method and degree counting arguments with pseudorandom properties of~$\Gnp$, 
in order to show the existence of many `well-behaved' maximal clique candidates in~$W$. 
To implement this plan, we first use the probabilistic method to randomly partition~$W$ into $m+1$ parts~$A,B_1, \ldots, B_m$, 
in a way that ensures that all vertices outside of~$W$ with many neighbors in~$A$ (say at least~$0.8|A|$ many) have zero neighbors in one of the parts~$B_1, \ldots, B_m$; see \refL{lem:partition}. 
For carefully chosen parameters~$m,k$ (see \refL{lem:choices} and \refS{sec:choices}), we then focus on the `structured' collection~$\cC$ of $k$-element subsets~$K \subseteq W$ whose vertices satisfy~$|K \cap A|=k-m$ and~$|K \cap B_i|=1$ for all $1 \leq i \leq m$, the point being that these~$K \in \cC$ are by construction not contained in the neighborhood of any vertex outside of~$W$ with many neighbors in~$A$. 
To implement step~(ii), the plan is 
to then show that~$\cC$ contains a large subcollection~$\cC' \subseteq \cC$ of maximal clique candidates, say at least~$|\cC'| \ge |\cC|/10$ many.
We would like to implement this via degree-counting arguments to estimate the number of `bad'~$K \in \cC$, i.e., where~$K$ is contained in the neighborhood of some vertex outside of~$W$. 
This works for reasonably sparse edge-probabilities~$p=p(n)$ by adapting the (level-set based) 
degree counting arguments of Lichev, Mitsche and Warnke~\cite{LMW2023}. 
But in order to go down to very sparse edge-probabilities~$p \gg n^{-2/5+\eps}$, we needed to develop a new degree counting argument (whose proof exploits codegree information)
to deal with vertices outside of~$W$ that have a reasonable number of neighbors in~$A$; see \refL{lem:counting} and property~\ref{lem:density:iii} in \refL{lem:density}. 
After these preparations, we are then in position to implement step~(iii): we show that at least one of the~$|\cC'| \ge |\cC|/10$ many maximal clique candidates~$K \in \cC'$ is a clique in~$\Gnp$, by combining Janson's inequality with a union bound argument;~see~\refL{lem:Janson}. 

Finally, in addition to the above-discussed difficulties, many of the technicalities in \refS{sec:main} stem from the fact that there is very little elbow room in the arguments for very sparse~$p$, including the `correct' choice of~$k$ and~$m$ (which are both fairly small constants in that range of~$p$, see \refS{sec:param:choices}). 
Furthermore, in order to obtain the correct asymptotics in \refT{thm:main3}, we needed to develop a new upper bound on the clique chromatic number that bootstraps existing upper bounds (to color certain  remaining `leftover' vertices); see~\refL{lem:upper-bound-epsilon-1}.

\subsection{Organization of the paper}
The remainder of this paper is organized as follows.
In the next subsection we state our main technical result, and show how it implies our main results \mbox{\refT{thm:main1}--\ref{thm:main3}}.
The proof of this technical result in Sections~\ref{sec:main}--\ref{sec:main:rem} uses some additional auxiliary results, whose standard proofs are given in~\refS{sec:deferred}. 
New upper bounds on $\chi_c(\Gnp)$ are given in \refS{sec:upper}.  
The final Section~\ref{sec:concluding} contains some concluding remarks and open problems.

\subsection{Main technical result: proofs of \refT{thm:main1}--\ref{thm:main3}}\label{sec:mainproof} 
The main technical result of this paper is Theorem~\ref{thm:max-clique} below. 
As we shall demonstrate in this subsection, it is key for establishing the lower bounds on the clique chromatic number~$\chi_c(\Gnp)$ in our main results \refT{thm:main1}--\ref{thm:main3} (the corresponding upper bounds are conceptually much simpler). 
More concretely, our goal is to show that we typically have~$\chi_c(\Gnp) > s$, 
where we defer our choice of the real-valued parameter~$\delta = \delta(n,p) \in (0,1)$~in 
\begin{equation}\label{def:s}
s = s(n,p,\delta) := \left \lfloor \delta\log_{\frac{1}{1-p}}(n)\right \rfloor .
\end{equation}
For this our starting point is the following auxiliary result, which formalizes certain pseudorandom properties of~$\Gnp$.    
The proof of \refL{lem:non-nbr:lb} is based on standard Chernoff bounds, and thus deferred to~\refS{sec:lem:non-nbr:lb}. 
As usual, we henceforth use whp as an abbreviation for with high probability (i.e., with probability tending to~one as~$n \to \infty$).
Recall that~$V=\{1, \ldots, n\}$ denotes the vertex set of~$\Gnp$ (as defined above~\eqref{eq:previous}). 
%
\begin{lemma}\label{lem:non-nbr:lb}
If~$\delta = \delta(n,p) \in (0,1)$ satisfies~$n^{1-\delta}p \gg \log^2(n)$, then there is~$\tau=\tau(n,p,\delta)=o(1)$ such that whp the following holds in~$\Gnp$. 
Every set~$S \subseteq V$ of size~$|S| \leq s$ has at least $(1 - \tau)n^{1-\delta}$ mutual non-neighbors, 
where~${s = s(n,p,\delta)}$ is defined as in~\eqref{def:s}. 
Furthermore, every vertex~$v \in V$ has at most~$2np$ neighbors. 
\end{lemma}
%
We 
say that a set~$W \subseteq V$ is \emph{useful} if $|V \setminus W| \ge \max\{s-1,1\}$ and every vertex $v \in V\setminus W$ has at~least
\begin{equation}\label{def:ell1}
\ell_1(W) =\ell_1(n,p,\delta,W) := \max\Bigcpar{(1-\tau)n^{1-\delta}/s, \; |W|-2np} 
\end{equation}
many non-neighbors in~$W$, where~${s = s(n,p,\delta)}$ is defined as in~\eqref{def:s} and~${\tau=\tau(n,p,\delta)}=o(1)$ is defined as in~\refL{lem:non-nbr:lb}. 
By taking the non-neighbors into account, note that any useful set~$W$ has size at least
\begin{equation}\label{def:ell0}
|W| \ge \ell_1(W) \ge \ell_1(\emptyset) =: \ell_0 .
\end{equation}
After these preparations we are now ready to state our main technical result \refT{thm:max-clique} below, whose proof is spread across the core Sections~\ref{sec:main} and~\ref{sec:deferred} of this paper. 
\begin{theorem}[Main technical result]\label{thm:max-clique}
Suppose that the assumptions of \refT{thm:main1} hold.
Then there is a choice of the real-valued parameter~$\delta = \delta(n,p,\eps)$ with~$\min\{\eps,1/200\} \le \delta \le 1/2$ such that 
(i)~the assumptions of \refL{lem:non-nbr:lb} hold, 
(ii)~the parameter~${s = s(n,p,\delta)}$ defined in~\eqref{def:s} satisfies~$s \ge 2$ for all sufficiently large~$n$, and 
(iii)~whp every useful set~$W \subseteq V$ in $G_{n,p}$ contains a clique~$K$ so that every vertex $v \in V\setminus W$ has at least one non-neighbor in~$K$.
Furthermore, we have~$\delta = 1/2+o(1)$ when~$n^{-o(1)} \le p \ll 1$. 
\end{theorem}
Ignoring a number of details, in the below proofs of our main results 
we shall use \refT{thm:max-clique} and \refL{lem:non-nbr:lb} to essentially argue as follows.  
If we partition the vertex set of~$\Gnp$ into any~$s$ disjoint non-empty color classes~$W_1, \ldots, W_s$, 
then (i)~at least one of the classes~$W_j$ is useful and (ii)~this useful class~$W_j$ must contain an inclusion-maximal clique. 
This in turn establishes that $\chi_c(\Gnp) \ge {s+1} \ge {\delta \log_{\frac{1}{1-p}}(n)}$, 
which by the properties of~$\delta$ asserted by \refT{thm:max-clique} then establishes the lower bounds 
in Theorem~\ref{thm:main1} and~\ref{thm:main2}. 
\begin{proof}[Proof of Theorem~\ref{thm:main1} and~\ref{thm:main2}]
The upper bounds in~\eqref{eq:main1} and~\eqref{eq:main2} follow from Lemma~3.3 and Theorem~3.4 in~\cite{MMP2019}; 
see also the proof of \refR{rem:upper-bound} in \refS{sec:upper} for a self-contained proof. 

We now turn to the more interesting lower bounds in~\eqref{eq:main1} and~\eqref{eq:main2}, for which we pick~$\delta = \delta(n,p,\eps) \in (0,1)$ as in \refT{thm:max-clique}. 
Since these hold whp, we henceforth assume that the random graph~$G_{n,p}$ satisfies all the properties and assertions of both~\refL{lem:non-nbr:lb} and~\refT{thm:max-clique}, 
and also assume that~$n$ is sufficiently large so that~$s \ge 2$ by the second assertion of \refT{thm:max-clique}. 

Aiming at a contradiction, suppose that~$\chi_c(\Gnp)  \le s$, where~$s=s(n,p,\delta)$ is as defined in~\eqref{def:s}. 
Hence there exists a valid clique coloring of the vertices of~$G_{n,p}$ with non-empty color classes~$W_1, \ldots, W_s$.   
For each color class~$W_i$, we denote by~$v_i \in V \setminus W_i$ some vertex with the fewest number of non-neighbors in~$W_i$ among all vertices in~$V \setminus W_i$. 
Let~$N$ denote the set of mutual non-neighbors of~$S:=\{v_1, \ldots, v_s\}$. 
By the first assertion of Lemma~\ref{lem:non-nbr:lb}, we know that~$|N| \ge (1-\tau)n^{1-\delta}$. 
By averaging, there exists a color class~$W_j$ with~$|W_j \cap N| \ge |N|/s$. 
By construction, all vertices in~$W_j \cap N$ are non-neighbors of~${v_j \in S \cap (V \setminus W_j)}$.
By the extremal choice of~$v_j$, it thus follows that every vertex~$v \in V\setminus W_j$ has at least~$|W_j \cap N| \ge |N|/s \ge (1-\tau)n^{1-\delta}/s$ non-neighbors in~$W_j$.
By the degree assertion of Lemma~\ref{lem:non-nbr:lb}, it also follows that every~$v \in V\setminus W_j$ has at least $|W_j| - 2np$ non-neighbors in~$W_j$. 
Combining these two lower bounds, in view of~\eqref{def:ell1} and~$|V \setminus W_j| = \sum_{i \in [s]: i \neq j}|W_i| \ge s-1 \ge 1$ we infer that~$W_j$ is useful. 
By the third assertion of \refT{thm:max-clique}, it then follows that~$W_j$ contains a clique~$K$ so that every~$v \in V\setminus W_j$ has at least one non-neighbor in~$K$. 
This clique is contained in some inclusion-maximal clique, say~$K^+$. 
Note that~$K^+$ is contained in~$W_j$, since every vertex~$v \in V\setminus W_j$ has at least one non-neighbor in~$K \subseteq K^+$. 
Hence~$W_j$ contains an inclusion-maximal clique, which contradicts that~$W_j$ is the color class of a valid clique coloring of~$G_{n,p}$. 

To sum up, we so far obtained that whp
\begin{equation}\label{eq:chi:lower}
\chi_c\bigpar{\Gnp} \ge s+1 \ge \delta \log_{\frac{1}{1-p}}(n) =(1-o(1)) \frac{\delta \log(n)}{p} ,
\end{equation}
where the last estimate exploits that~$p \ll 1$ implies~$\log\bigpar{\tfrac{1}{1-p}}=-\log(1-p) = p(1+o(1))$.  
Note that~\eqref{eq:chi:lower} establishes the lower bounds in~\eqref{eq:main1} and~\eqref{eq:main2}, 
since \refT{thm:max-clique} guarantees that~$\delta \ge \min\{\eps,1/200\}= \Omega(1)$ always holds, in addition to the more precise estimate~$\delta = 1/2+o(1)$ when~${n^{-o(1)} \le p \ll 1}$. 
\end{proof}

For sufficiently small~$\eps>0$, the above proof gives the lower bound~$\delta \ge \eps$ in~\eqref{eq:chi:lower}. 
This is not a proof artifact, since for~$p=n^{-2/5+\eps}$ a linear dependence of~$\delta$ on $\eps$ is in fact needed, 
as shown by~\refL{lem:upper-bound-epsilon-1} in \refS{sec:upper}.
In the following proof of \refT{thm:main3} we take this observation one step further, 
and use the following refinement of \refT{thm:max-clique} (whose proof is given in \refS{sec:main:rem}) 
to determine the typical asymptotics of $\chi_c(\Gnp)$ for~$p=n^{-2/5+\eps}$, where we even allow for~$\eps=\eps(n) \to 0$ at some~rate. 
\begin{theorem}[Refinement of \refT{thm:max-clique} in very sparse case]\label{rem:max-clique}
Suppose that the assumptions of \refT{thm:main3} hold for ${p=n^{-2/5+\eps}}$, with~$\eps=\eps(n) > 0$. 
Then the conclusion of \refT{thm:max-clique} remains valid, with~$\min\{\eps,1/200\} \le \delta \le 1/2$ replaced by~${\delta = (1+o(1))5 \eps/2}$.  
\end{theorem}
\begin{proof}[Proof of Theorem~\ref{thm:main3}]
We start with the lower bound in~\eqref{eq:main3}, for which we pick~$\delta = \delta(n,p,\eps) \in (0,1)$ as in \refT{rem:max-clique}.
By proceeding  word-by-word as in the above proof of Theorem~\ref{thm:main1} and~\ref{thm:main2}, 
for this choice of~$\delta$ we then establish that whp the lower bound~\eqref{eq:chi:lower} holds. 
Since \refT{rem:max-clique} guarantees that~${\delta = (1+o(1))5 \eps/2}$,
this completes the proof of the lower bound in~\eqref{eq:main3}, by noting that for~$p=n^{-2/5+\eps}$ we have~$\eps \log(n) = \log(n^{2/5}p)$.

Finally, the upper bound in~\eqref{eq:main3} follows from \refL{lem:upper-bound-epsilon-1} in \refS{sec:upper}, since~$\eps \log(n) = \log(n^{2/5}p)$.  
\end{proof}

\section{Proof of main technical result: \refT{thm:max-clique}}\label{sec:main}
In this section we prove our main technical result \refT{thm:max-clique}, by proceeding in three steps (where $V=\{1, \ldots, n\}$ denotes the vertex set of~$\Gnp$, as before). 
First, in \refS{sec:deterministic} we use combinatorial arguments to show that any useful set~$W \subseteq V$ contains a large collection~$\cC$ of \emph{inclusion-maximal clique candidates (with respect to~$W$)}, 
i.e., sets~$K \subseteq W$ for which every vertex~$v \in V\setminus W$ has at least one non-neighbor in~$K$. 
Second, in \refS{sec:janson} we use probabilistic arguments to show that at least one of these clique candidates~$K \in \cC$ in fact is a clique in~$\Gnp$. 
Third, in \refS{sec:choices} we establish the desired form of~$\delta$ asserted by \refT{thm:max-clique}, and verify several other technical~inequalities needed for our arguments. 

We now turn to the technical details of our three-step approach, 
which requires us to introduce a number of different auxiliary variables. 
With foresight, we define $\phi(x) := (1+x)\log(1+x) - x$ as well as
\begin{equation}\label{eq:param}
x_i := \frac{\log(n)}{\phi(r_i - 1)p}, \qquad 
r_i := e^{\gr i}, \qquad 
\gr : = \frac{1}{\log^{4}(n)}, \qquad 
\text{and} \qquad
\rho := \log_n(1/p).
\end{equation}
These variables must satisfy a number of technical inequalities for our arguments to work, 
which requires different parameter choices for different ranges of edge-probabilities~$p=p(n)$. 
\refL{lem:choices} below records that we can always choose the parameters in a suitable way. 
We defer the technical proof of~\refL{lem:choices} to \refS{sec:choices}, 
which in the upcoming Sections~\ref{sec:deterministic} and~\ref{sec:janson} then 
allows us to focus on our main combinatorial and probabilistic line of reasoning. 
On a first reading, the details of the technical inequalities~\eqref{ass:Janson:k:bound}--\eqref{eq:m1:delta} below can safely be ignored. 
%
\begin{lemma}[Choice of parameters]\label{lem:choices}%
Suppose that the assumptions of \refT{thm:main1} hold. 
Set ${\portion := 1/10}$, ${\sigma := 1/100}$ and ${\alpha:=4/5}$. 
Then there is a choice (that depends on~$n,p,\eps$) of the integer-valued parameters $k,m \ge 1$ and the real-valued parameter $\delta \in (0,1)$ that satisfy, for all sufficiently large~$n$, 
the following inequalities: 
\begin{gather}
\label{ass:Janson:k:bound}
m +2 \le k \leq \log(n), \\ 
\label{ass:Janson-delta-bound}
\delta \le \min\biggcpar{\frac{1}{2} - \rho, \ \frac{k-1}{k}\Bigpar{1-\rho(k/2+1)}} -\frac{9 \log(\log(n))}{\log(n)} , \\
\label{ass:partition}
\min\Bigcpar{\ell_0, \: n^{1-\delta}p} \geq \max\Bigcpar{16m\bigpar{1+\log(n m)} , \ 8k^2, \ \log^4(n)},\\
\label{eq:density:1}
\min_{i \ge 1} \: \ceil{x_i}\Bigpar{\phi(r_i-1)p - \log^3(n)/\ell_0} \geq 1+\max\Bigcpar{\log\bigpar{n\log^2(n)e/\ell_0}, \ 0} ,\\
\label{ass:counting}
1-\Lambda \ge \portion ,\\
\label{ass:s}
s\ge m+1
\end{gather}
where~$\Lambda=\Lambda(n,p,k,m) \ge 0$ 
is defined as in~\eqref{eq:def:Lambda} of the proof of~\refL{lem:counting}, 
${s = s(n,p,\delta)}$ is defined as in~\eqref{def:s},
and~$\ell=\ell_1(\emptyset)$ is defined as in~\eqref{def:ell1}.  
Furthermore, 
the following is also true: if~$m \ge 2$, then 
\begin{gather}
\label{eq:p:large}
p \ge n^{-\sigma},\\
\label{eq:m2:delta}
\delta = 1/2 - 3\rho - \frac{9 \log(\log(n))}{\log(n)},\\
\label{eq:density:2}
(m+1) \Bigpar{\phi(\alpha/p-1)p - \log^3(n)/\ell_0} \geq 1+\max\Bigcpar{\log\bigpar{n\log^2(n)e/\ell_0}, \ 0} ,
\end{gather}
and if~$m=1$, then 
\begin{gather}
\label{eq:p:small}
p < n^{-\sigma},\\
\label{eq:m1:delta}
\delta = \min\{\eps, \: \sigma/2\} . 
\end{gather}%
\end{lemma}

As discussed, we henceforth fix the choices of the parameters~$k,m \ge 1$ and $\delta,\portion,\sigma,\alpha \in (0,1)$ as given by \refL{lem:choices}, 
which by~\eqref{def:s} also fixes the parameter~$s$. 
In the remainder of this section we can thus always tacitly assume all the properties and assertions of \refL{lem:choices} (whenever needed). 
Note that by~\eqref{ass:partition} we have $n^{1-\delta}p \geq \log^4(n) \gg \log^2(n)$, so the assumptions of~\refL{lem:non-nbr:lb} hold.
By~\eqref{ass:s} we also have~$s \ge m+1 \ge 2$.
Furthermore, inspecting the definitions~\eqref{eq:m2:delta} and~\eqref{eq:m1:delta} for the different ranges~\eqref{eq:p:large} and~\eqref{eq:p:small} of~$p=p(n)$, 
in view of~$\sigma = 1/100$ it is easy to check that $\min\{\eps, 1/200\} \le \delta \le 1/2$ for all sufficiently large~$n$. 
Note that for~$n^{-o(1)} \le p \ll 1$ we also have~$\rho=\log_n(1/p)=o(1)$, and thus~$\delta = 1/2+o(1)$ by~\eqref{eq:m2:delta}. 
Putting things together, so far we have verified all properties of~$\delta$ asserted by \refT{thm:max-clique}, including its assertions~(i) and~(ii). 

To complete the proof of \refT{thm:max-clique}, it therefore remains to establish its assertion~(iii), i.e., that whp every useful set~$W \subseteq V$ in $G_{n,p}$ contains a clique~$K$ so that every vertex $v \in V\setminus W$ has at least one non-neighbor in~$K$. 
The proof of assertion~(ii) is spread across the remainder of this section, and naturally splits into three parts: a combinatorial part in \refS{sec:deterministic}, a probabilistic part in \refS{sec:janson}, and a  technical part in \refS{sec:choices}.

\subsection{Combinatorial part: existence of many maximal clique candidates}\label{sec:deterministic}
The goal of this section is to prove that any useful set~$W \subseteq V$ contains a large collection~$\cC$ of inclusion-maximal clique candidates. 
Our combinatorial arguments will exploit some additional pseudorandom properties of~$\Gnp$, which are formalized in the following important auxiliary result, where~$\Gamma(v)=\Gamma_{\Gnp}(v)$ denotes the set of neighbors of vertex~$v$ in~$\Gnp$. 
The proof of \refL{lem:density} is based on standard Chernoff bounds, and thus deferred to~\refS{sec:lem:density} 
(as we shall see, the proof of property~\ref{lem:density:iii} is somewhat roundabout, but key for very small edge-probabilities). 
Recall that $V=\{1, \ldots, n\}$ denotes the vertex set of~$\Gnp$, 
that whp is an abbreviation for with high probability, 
and that $\ell_0 = \ell_1(\emptyset)$ is defined as in~\eqref{def:ell0}.
Furthermore, since we fixed the integer~$m \ge 1$ according to \refL{lem:choices}, 
in view of~\eqref{eq:p:large} and~\eqref{eq:p:small} the two cases~$m \ge 2$ and~$m=1$ below effectively two distinguish different ranges of~$p=p(n)$. 
\begin{lemma}
\label{lem:density}
A vertex set~$W \subseteq V$ is called \emph{nice}
if all vertex subsets~$U \subseteq W$ of size at least~$|U| \ge \ell_0/\log^2(n)$ satisfy the event~${\cD_{U,W}}$, 
which consists of the following three assertions: 
\vspace{-0.25em}{\begin{romenumerate}
	\parskip 0em  \partopsep=0pt \parsep 0em 
	\item\label{lem:density:i} For each integer~$i \ge 1$, there are at most $x_i$ many vertices~$v \in V \setminus W$ with $|\Gamma(v) \cap U| \ge r_i p |U|$. 
 \item\label{lem:density:ii} If~$m \ge 2$, then there are at most $m$ many vertices~$v \in V \setminus W$ 
 with $|\Gamma(v) \cap U| \ge \alpha|U|$. 
	\item\label{lem:density:iii} 
 If~$m=1$, then, for each integer~$1 \le j \le 9\log(n)$, there are at most~$j$ many vertices~$v \in V \setminus W$ with $|\Gamma(v) \cap U| \ge \bigpar{1/(j+1) + 1/\log^3(n)} \cdot |U|$.
\end{romenumerate}}\vspace{-0.25em}
Then whp all vertex subsets~$W \subseteq V$ of~$\Gnp$ are nice. 
\end{lemma}

Turning to the details, in the next two subsections we fix one useful and nice set~$W \subseteq V$, and deterministically show that~$W$ contains a large collection~$\cC$ of inclusion-maximal clique candidates (with respect to~$W$) of size~$k$,
i.e., sets~$K \subseteq W$ of size~$|K|=k$ for which every vertex~$v \in V\setminus W$ has at least one non-neighbor in~$K$ 
(where the integer~$k$ is chosen according to \refL{lem:choices}, as discussed).

\subsubsection{Constructing a pseudo-partition of~$W$ (probabilistic method)}\label{sec:step1} 
Intuitively, our first step is to construct a pseudo-partition of~$W$, see \refL{lem:partition} below.  
Let us motivate why this helps us  to find many inclusion-maximal clique candidates. 
Namely, one could perhaps hope to argue that a typical $k$-subset~$K \subseteq W$ is an inclusion-maximal clique candidate, but this turns out to be too optimistic: the problem is that the `high-degree' vertices of~${V \setminus W}$ are adjacent to most of $W$, in which case a typical $k$-subset $K \subseteq W$ is likely to be completely contained in the neighborhood of such vertices (and thus not an inclusion-maximal clique candidate).
To get around this major obstacle the idea is that, for each such `high-degree' vertex~$u_i$, we should try to have at least one vertex in our clique-candidate~$K$ that is 
a non-neighbor of~$u_i$. 
It turns out that by randomly partitioning most of $W$ into several disjoint parts, with certain parts corresponding to non-neighbors of a fixed  `high-degree' vertex~$u_i$, we are eventually  able to make this idea work and construct many inclusion-maximal clique candidates. 
The following lemma is the first step towards making this plan precise, 
where we slightly abuse notation in~\eqref{prop:A}--\eqref{prop:Bi} below, since the parameters~$a$ and~$b$ both depend~on~$|W|$. 
%
%
Recall that $\ell_1(W)$ is defined as in~\eqref{def:ell1}. 
\begin{lemma}\label{lem:partition}
There exists disjoint subsets $A,B_1, \ldots, B_m$ of $W$ with 
\begin{align}
\label{prop:A}
|A| &= {\bigl\lceil |W|/4 \bigr\rceil} =: a,\\
\label{prop:Bi}
|B_i| & = {\bigl\lceil\ell_1(W)/(4m)\bigr\rceil} =: b \qquad \text{for all~$i \in [m]$}
\end{align}
such that~$|\Gamma(u_i) \cap B_i|=0$ for all~$i \in [m]$, where $u_1, u_2, \ldots$ is an enumeration of all vertices in $V\setminus W$ in decreasing order according to their number of neighbors~$|\Gamma(u_j) \cap A|$ in~$A$ (ties broken in lexicographic order, say). 
\end{lemma}
\begin{proof}%
We first randomly partition~$W$ into $A^+ \cup B_1^+ \cup \cdots \cup B_m^+$, 
by independently placing each vertex~$v \in W$ into~$A^+$ with probability~$1/2$,  and into~$B_i^+$ with probability~$1/(2m)$ for all $i \in [m]$.
Note that the marginal distribution of~$|A^+|$ equals the Binomial distribution~$\Bin\bigpar{|W|,1/2}$. 
Similarly, for each pair of~$i \in [m]$ and~$v \in V\setminus W$, the marginal distribution of~$|B^+_i \setminus \Gamma(v)|$ equals~$\Bin\bigpar{|W \setminus \Gamma(v)|,1/(2m)}$. 
Recall that since~$W$ is useful, 
by~\eqref{def:ell0} we have~$|W| \ge \ell_0$. 
By definition of useful (given below \refL{lem:non-nbr:lb}), we also have $|W \setminus \Gamma(v)| \ge \ell_1(W) \ge \ell_0$ for every $v \in V\setminus W$, 
as well as $|V \setminus W| \ge s-1 \ge m$ by assumption~\eqref{ass:s}. 
Using a standard union bound argument and Chernoff bounds (see~\cite[Theorem~2.1]{JLR}),
the probability that one of the random variables in the~collection 
\[
\bigcpar{|A^+|} \cup \Bigpar{\bigcup_{i \in [m]} \bigcup_{v \in V \setminus W} \bigcpar{|B^+_i \setminus \Gamma(v)|}}
\]
is less than half of its expected value is at~most
\begin{equation}\label{eq:partition:error1}
e^{-|W|/16} \: + \: \sum_{i \in [m]} \sum_{v \in V\setminus W} e^{-|W \setminus \Gamma(v)|/(16m)} \le e^{-\ell_0/16} + m \cdot n \cdot e^{-\ell_0/(16m)} \le 2 \cdot e^{-1}  < 1 ,
\end{equation}
where the second last inequality holds by assumption~\eqref{ass:partition}. 
Since the failure probability~\eqref{eq:partition:error1} is strictly less than one, 
by the probabilistic method there must exist a partition $A^+ \cup B_1^+ \cup \cdots \cup B_m^+$ of~$W$ with
\begin{align}
\label{prop:A:plus}
|A^+| &\ge |W|/4,\\
\label{prop:Bi:plus}
|B^+_i \setminus \Gamma(v)| &\ge |W \setminus \Gamma(v)|/(4m) \ge \ell_1(W)/(4m) \qquad \text{for all~$i \in [m]$ and $v \in V \setminus W$.}
\end{align}

To obtain the desired disjoint subsets $A,B_1, \ldots, B_m$ of~$W$, we now further refine the above-constructed partition $A^+,B^+_1 \ldots, B^+_m$ of~$W$ as follows. 
First, using~\eqref{prop:A:plus}  we pick an arbitrary subset~$A \subseteq A^+$ of size~$|A|=a$, 
which determines the enumeration~$u_1, u_2, \ldots$ of the vertices in~$V \setminus W$ appearing in the statement of \refL{lem:partition}. 
Afterwards, using~\eqref{prop:Bi:plus} and $|V \setminus W| \ge m$  (as established above) we then pick an arbitrary subset~$B_i \subseteq B_i^+ \setminus \Gamma(u_i)$ of size~$|B_i|=b$ for each~$i \in [m]$, completing the proof of \refL{lem:partition}. 
\end{proof}

\subsubsection{Counting maximal clique candidates in pseudo-partition of~$W$ (degree counting)}\label{sec:step3}   
Using the pseudo-partition $A,B_1, \ldots, B_m$ of~$W$ given by \refL{lem:partition}, our second step is to then show that there are many inclusion maximal clique candidates in 
\begin{equation}\label{eq:almost-cc}
\acc := \bigcpar{K \subseteq W \; : \; |K|=k, \: |K \cap A|=k-m \text{ and } |K \cap B_i|=1 \text{ for all~$i \in [m]$}} ,
\end{equation}
as formalized by \refL{lem:counting} below
(where we again slightly abuse notation, since~$\cc$ and~$\acc$ both depend on~$W$). 
Here our strategy is to count the number of~${K \in \acc}$ that are not inclusion-maximal clique candidates, i.e., for which~$K$ is contained in the neighborhood of some vertex in~$V \setminus W$.
The point is that we can estimate the number of such `bad'~${K \in \acc}$ via degree counting arguments, 
though the details are significantly more involved than in earlier work~\cite{AK2018,LMW2023}. 
Here one difficulty is that~$|W|$ can be much smaller than the typical degree~$np$, 
which means that a single vertex could potentially make all~$K \in \acc$ bad. 
We overcome this obstacle by combining combinatorial properties of $A,B_1, \ldots, B_m$ with pseudorandom properties from \refL{lem:density}, which allows us to deduce that all vertices in~$V \setminus (W \cup \{u_1, \ldots, u_m\})$ have at most~$\alpha |A|$ neighbors in~$A$, say, where~$u_1, \ldots, u_m$ are the $m$ highest-degree vertices given by the enumeration from \refL{lem:partition}.  
However, the major technical difficulty is that  natural degree counting arguments that exploit property~\ref{lem:density:i} of the event~$\cD_{U,W}$ (see \refL{lem:density} as well as~\eqref{eq:accX:0}--\eqref{eq:accX:1} and~\eqref{eq:Pi:alpha}--\eqref{eq:Pi:alpha:3} below) yield `too large' upper bounds on the number of bad~$K \in \acc$ for very small~$p=p(n)$. 
We overcome this major obstacle by developing a new degree counting argument in the sparse case~$m=1$,  based on property~\ref{lem:density:iii} of the event~$\cD_{A,W}$ (see~\eqref{eq:accX:10}--\eqref{eq:accX:11} below). 
Recall that the constant~${\portion = 1/10}$ is given by \refL{lem:choices}. 
Note that~${|\acc| = \binom{a}{k-m}b^m}$, where~$a$ and~$b$ are defined as in~\eqref{prop:A}--\eqref{prop:Bi}. 
\begin{lemma}\label{lem:counting}
There exists~$\cc \subseteq \acc$ containing at least~$|\cc| \geq \portion|\acc|$ many inclusion-maximal clique candidates with respect to~$W$. 
\end{lemma}
\begin{proof}
Writing~$\deg_S(v) := |\Gamma(v) \cap S|$ for the number of neighbors of~$v$ in a set~$S$, we partition~$V \setminus W$ into 
\begin{equation*}
V\setminus W= L \cup X \cup Y \cup Z ,
\end{equation*}
where~$L:=\{u_1, \ldots, u_m\} \subseteq V \setminus W$ contains the $m$ highest-degree vertices from \refL{lem:partition}, and
\begin{equation*}
\begin{split}
X :=& \bigcpar{v \in V\setminus (W \cup L) \: : \:   \deg_A(v) \ge r_1p|A| },\\
Y :=& \bigcpar{v \in V\setminus (W \cup L) \: : \: \deg_A(v) < r_1p|A| \text{ and } \deg_{B_i}(v) \geq r_1p|B_i| \text{ for some }i \in [m]},\\
Z :=& \bigcpar{v \in V\setminus (W \cup L) \: : \: \deg_A(v) < r_1p|A| \text{ and } \deg_{B_i}(v) < r_1p|B_i| \text{ for all }i \in [m]}.
\end{split}
\end{equation*}
 We define $\acc_X$, $\acc_Y$, and $\acc_Z$ as the sets~$K \in \acc$ that are contained in the neighborhood of some vertex in $X$, $Y$, and $Z$, respectively. 
 Observe that, by \refL{lem:partition} and the definition~\eqref{eq:almost-cc} of~$\acc$, every set $K \in \acc$ contains at least one non-neighbor of every vertex~$u_j \in L$.
 Hence all sets in
\begin{equation}\label{eq:PXYZ:cc}
 \cc := \acc \setminus (\acc_X \cup \acc_Y \cup \acc_Z)
\end{equation}
are inclusion-maximal clique candidates with respect to~$W$. 
To obtain the desired lower bound on~$|\cc|$, 
it thus suffices to derive suitable upper bounds on 
$|\acc_X|$, $|\acc_Y|$ and $|\acc_Z|$. 

We start with an upper bound on~$|\acc_Y|$.  Note that, using $\ell_1(W) \ge \ell_1(\emptyset)=\ell_0$ and~\eqref{ass:Janson:k:bound}, we have $|B_i| = \ceil{\ell_1(W)/(4m)} \ge {\ell_1(\emptyset)/(4m)} \gg {\ell_0/\log^2(n)}$ for every~$i \in [m]$. 
We now exploit that every vertex~$v \in Y$ satisfies $\deg_{B_i}(v) \ge r_1 p |B_i|$ for some~$i \in [m]$. 
Indeed, since $W$ is nice,
by property~\ref{lem:density:i} of the event~$\cD_{B_i,W}$ (see \refL{lem:density}) 
we infer that~$|Y| \leq m \cdot x_1$. 
For every vertex~$v \in Y$, observe that there are $\binom{\deg_A(v)}{k-m}$ many subsets of~$A$ of size~${k-m}$ in the neighborhood of~$v$. 
By exploiting that every vertex~$v \in Y$ satisfies $\deg_A(v) < r_1 p |A|$, it follows~that
\begin{align}\label{eq:accY}
\frac{|\acc_Y|}{|\acc|}& \leq \sum_{v \in Y}
\underbrace{\frac{\binom{\floor{r_1p|A|}}{k-m}}{\binom{|A|}{k-m}}}_{\le (r_1p)^{k-m}} \cdot \underbrace{\prod_{i \in [m]} \frac{|B_i|}{|B_i|}}_{=1}\leq |Y| \cdot (r_1 p)^{k-m} \le mx_1 \cdot (r_1 p)^{k-m} .
\end{align}

Next we estimate~$|\acc_Z|$ from above.  By definition every vertex in~$Z$ is adjacent to at most an~$r_1p$ fraction of each of $A,B_1,\ldots,B_m$. 
By proceeding similarly to the estimate~\eqref{eq:accY} for~$|\acc_Y|$, 
using~$|Z| \le n$ it follows~that  
\begin{align}\label{eq:accZ}
\frac{|\acc_Z|}{|\acc|}
\le \sum_{v \in Z}
\underbrace{\frac{\binom{\floor{r_1p|A|}}{k-m}}{\binom{|A|}{k-m}}}_{\le (r_1p)^{k-m}} \underbrace{\prod_{i \in [m]} \frac{r_1 p |B_i|}{|B_i|}}_{=(r_1 p)^m} 
\le |Z|\cdot (r_1 p)^k \leq n \cdot (r_1 p)^k .
\end{align}

Finally we estimate~$|\acc_X|$, which is the crux of the matter. 
For any real-valued parameter~$\beta \in (0,1)$, we~define
\begin{equation*}
X_\beta := \{v \in X~:~\deg_A(v) \leq \beta|A|\}.
\end{equation*}
Recall that $|A| = \ceil{|W|/4} \geq \ell_1(\emptyset)/4 \gg \ell_0/\log^2(n)$. 
We now exploit that $L =\{u_1, \ldots, u_m\} \subseteq V \setminus W$ contains the $m$ highest-degree vertices from the enumeration in the statement of \refL{lem:partition}, using a case distinction. 

\vskip2mm

\textbf{Case $m \ge 2$: }  
Since $W$ is nice,
by property~\ref{lem:density:ii} of the event~$\cD_{A,W}$ (see \refL{lem:density}) 
we infer that all vertices $v \in X \subseteq V\setminus (W \cup L)$ must satisfy~$\deg_A(v) < \alpha |A|$. 
This degree bound establishes that~$X = X_\alpha$, so by proceeding similarly to~\eqref{eq:accY} and~\eqref{eq:accZ} it follows that 
\begin{align}\label{eq:accX:0}
\frac{|\acc_X|}{|\acc|} \leq \sum_{v \in X_{\alpha}} \frac{\binom{\deg_A(v)}{k-m}}{\binom{|A|}{k-m}}.
\end{align}
For each integer~$i \ge 1$, let $z_i$ denote the number of vertices $v \in X_\alpha$ with $\deg_A(v) \ge r_ip|A|$. 
Since~$W$ is useful, by property~\ref{lem:density:i} of the event~$\cD_{A,W}$ we know that~$z_i \le x_i$ for all~$i \geq 1$. 
Exploiting that all vertices~$v \in X_\alpha$ satisfy~$r_1 p |A| \le \deg_A(v) \le \alpha|A|$, by splitting the degrees into level-sets (also called dyadic decomposition) 
we conclude that 
\begin{equation}\label{eq:accX:1}
\begin{split}
\frac{|\acc_X|}{|\acc|} & \leq \sum_{i \ge 2: p r_{i-1} \le \alpha} (z_{i-1} - z_i) \cdot 
\underbrace{\frac{\binom{\floor{\min\{r_ip,\alpha\}|A|}}{k-m}}{\binom{|A|}{k-m}}}_{\le (r_ip)^{k-m}}\\ 
& \le z_1(r_2p)^{k-m} + \sum_{i \ge 3: r_{i-1}p  \le \alpha} z_{i-1} \cdot \Bigsqpar{(r_{i}p)^{k-m} - (r_{i-1}p)^{k-m}}\\
& \le x_1(r_2p)^{k-m} + \underbrace{\sum_{i \ge 2: r_{i} p  \le \alpha} x_i\Bigsqpar{(r_{i+1}p)^{k-m} - (r_{i}p)^{k-m}}}_{=: \Pi_{\alpha}}. 
\end{split}
\end{equation}
(Although not directly relevant here, 
we remark that~$\Pi_\alpha$ is intuitively dominated by the last term of its sum: 
the rigorous analysis and simplification appears in~\eqref{eq:Pi:alpha}--\eqref{eq:Pi:alpha:3} of \refS{sec:Lambda}, 
as part of the deferred proof of inequality~\eqref{ass:counting} from \refL{lem:choices}.)

\vskip2mm

\textbf{Case $m =1$: } 
It turns out that the previous bound~\eqref{eq:accX:1} on~$|\acc_X|/|\acc|$ degenerates for very small~$p=p(n)$, i.e., is no longer less than one (see estimates~\eqref{eq:Pi:alpha:3}--\eqref{eq:Pi:alpha:second} in \refS{sec:Lambda}, where~$\Pi_\alpha < 1 $ only holds when~$\rho=\log_n(1/p)$ is sufficiently small). 
To overcome this technical bottleneck, we thus had to develop another bound on~$|\acc_X|/|\acc|$, 
which exploits property~\ref{lem:density:iii} of~$\cD_{A,W}$ as well as $X\subseteq V\setminus (W \cup L)$ and~${L=\{u_1\}}$.
Turning to the details, by first decomposing~$X$ into~$X_{1/\log(n)}$ and~$X \setminus X_{1/\log(n)}$ 
and then proceeding similarly to~\eqref{eq:accX:0}--\eqref{eq:accX:1}, we obtain~that 
\begin{equation}\label{eq:accX:10}
\begin{split}
\frac{|\acc_X|}{|\acc|} & \leq \sum_{v \in X_{1/\log(n)}} \frac{\binom{\deg_A(v)}{k-m}}{\binom{|A|}{k-m}} + \sum_{v \in X \setminus X_{1/\log(n)}} \underbrace{\frac{\binom{\deg_A(v)}{k-m}}{\binom{|A|}{k-m}}}_{\le \bigpar{\deg_A(v)/|A|}^{k-m}}  \\
& 
\le x_1(r_2p)^{k-m} + \Pi_{1/\log(n)} + \sum_{v \in X \setminus X_{1/\log(n)}} \biggpar{\frac{\deg_A(v)}{|A|}}^{k-m},
\end{split}
\end{equation}
where the parameter~$\Pi_{1/\log(n)}$ is defined analogously to~$\Pi_{\alpha}$  in~\eqref{eq:accX:1}, but with~$r_i p \le \alpha$ replaced by~$r_i p \le 1/\log(n)$. 
Let $u_1, u_2, \ldots$ be the enumeration of all vertices in $V\setminus W$ as in \refL{lem:partition}, i.e., in decreasing order according to~$\deg_A(u_j)$.
By property~\ref{lem:density:iii} of the event~$\cD_{A,W}$, for any~$1 \le j \le 9\log(n)$ we know that at most $j$ many  vertices $v \in V \setminus W$ satisfy $\deg_A(v)/|A| \ge 1/(j+1)+ 1/\log^3(n)$.
The crux is that this implies 
\[
\frac{\deg_A(u_j)}{|A|} \le \frac{1}{j} + \frac{1}{\log^3( n)} \qquad \text{ for all~$u_j \in V \setminus W$ with $2 \le j \le 9 \log(n)$}.
\]
Since ${L=\{u_1\}}$ and~$X \subseteq V \setminus (W \cup L)$, we also infer that~$X \setminus X_{1/\log(n)} \subseteq \bigcpar{ u_j \in V \setminus W: \: 2 \le j \le 9 \log(n)}$, say. 
As assumption~\eqref{ass:Janson:k:bound} implies~$k \le \log(n)$ and~$k-m \ge 2$, it therefore follows that    
\begin{equation}\label{eq:accX:11}
\begin{split}
\sum_{v \in X \setminus X_{1/\log(n)}} \biggpar{\frac{\deg_A(v)}{|A|}}^{k-m}
& 
   \le \sum_{2 \le j \le 9 \log(n)} \lrpar{\frac{1}{j} + \frac{1}{\log^3( n)}}^{k-m} \\
&   = \sum_{2 \le j \le 9\log(n)} \lrpar{\frac{1}{j}}^{k-m} \lrpar{1 + \frac{j}{\log^3( n)}}^{k-m} \\
& \le (1+o(1)) \cdot \sum_{j \ge 2} \frac{1}{j^2} = \frac{\pi^2}{6}-1 +o(1) < 0.7. 
\end{split}
\end{equation}

\vskip2mm

To sum up: by combining the above estimates~\eqref{eq:accY}, \eqref{eq:accZ} for~$|\acc_Y|$ and~$|\acc_Z|$
with the two different estimates~\eqref{eq:accX:1} and~\eqref{eq:accX:10}--\eqref{eq:accX:11} for~$|\acc_X|$, using~$r_1 \le r_2$ it follows that 
\begin{equation}\label{eq:def:Lambda}
\frac{|\acc_X|  + |\acc_Y| + |\acc_Z|}{|\acc|} \le (m+1) x_1(r_2 p)^{k-m}  + n (r_1 p)^k + \indic{m \ge 2} \Pi_{\alpha} + \indic{m=1}\Bigsqpar{\Pi_{1/\log(n)} + 0.7} = : \Lambda,
\end{equation}
which by definition~\eqref{eq:PXYZ:cc} of~$\cc$ and 
assumption~\eqref{ass:counting} establishes the desired lower bound~$|\cc|/|\acc| \ge 1-\Lambda \ge \portion$, 
completing the proof of \refL{lem:counting}. 
(Although not directly relevant for this proof, 
we remark that~$\Lambda$ is further estimated in \refS{sec:Lambda}, as part of the deferred proof of inequality~\eqref{ass:counting} from \refL{lem:choices}:
it turns out that the first two terms of~$\Lambda$ are~$o(1)$, 
and that $\Pi_\alpha$ and $\Pi_{1/\log(n)}$ are small enough to ensure that $1-\Lambda \ge 1/10=\portion$.) 
\end{proof}

\subsection{Probabilistic part: existence of a maximal clique (Janson's inequality)}\label{sec:janson} 
In this section we use probabilistic arguments to prove assertion~(ii) of \refT{thm:max-clique}, i.e., that whp every useful set~$W \subseteq V$ in $G_{n,p}$ contains a clique~$K$ so that every vertex $v \in V\setminus W$ has at least one non-neighbor in~$K$ 
(as discussed in \refS{sec:mainproof}, the crux is that each such~$K$ is contained in some inclusion maximal clique~$K^+ \subseteq W$). 
The starting point here is \refL{lem:density}, which guarantees that whp every useful set is also nice. 
This allows us to bring \refL{lem:counting} into play, which guarantees that any useful and nice set~$W$ contains a large collection~$\cc$ of inclusion-maximal clique candidates~$K$ of size~$k$, 
i.e., sets~$K \subseteq W$ of size~$|K|=k$ for which every vertex~$v \in V\setminus W$ has at least one non-neighbor in~$K$. 
To establish assertion~(ii), we then use Janson's inequality~\cite{SJ1990,RW2015,JW2016} 
to prove that typically at least one of these clique candidates~$K \in \cc$ is in fact a clique in~$\Gnp$, as formalized by the following~lemma.
Recall that whp is an abbreviation for with high probability, and that $V=\{1, \ldots, n\}$ denotes the vertex set of~$\Gnp$.  
\begin{lemma}\label{lem:Janson}
Whp every useful set~$W \subseteq V$ in $G_{n,p}$ satisfies~$Y_W \ge 1$, 
where the random variable~$Y_W$ denotes the number of cliques~$K$ in~$W$ so that every vertex $v \in V \setminus W$ has at least one non-neighbor in~$K$. 
\end{lemma}
\begin{proof}
We start with several auxiliary estimates, that all hold for all sufficiently large~$n$.  
Recall that by~\eqref{def:ell0}, any useful set~$W \subseteq V$ must contain at least~$|W| \ge \ell_1(W)$ many vertices. 
Since~$s \le \log (n)/p$ as in the proof of~\refL{lem:non-nbr:lb} (see \refS{sec:lem:non-nbr:lb}), in view of the definition~\eqref{def:ell1} of~$\ell_1(W)$ and $1-\tau=1-o(1)$ it follows that 
\begin{equation}\label{eq:l1W:lower}
\begin{split}
\ell_1(W) \ge \ell_1(\emptyset) \ge (1-\tau) n^{1-\delta}/s \ge \tfrac{1}{2} n^{1-\delta}p \log^{-1}(n),
\end{split}
\end{equation}
and, using a basic case distinction, it also follows that
\begin{equation}\label{eq:l1W:ratio}
\begin{split}
\frac{\ell_1(W)} {|W|}
& \ge \indic{|W| \le 4np}\frac{(1-\tau)n^{1-\delta}/s}{|W|} + \indic{|W| \ge 4 np} \frac{|W| - 2np}{|W|} \\
& \ge \indic{|W| \le 4np}\frac{(1-\tau)n^{-\delta}}{4p s} + \indic{|W| \ge 4 np}(1-1/2) \ge \tfrac{1}{8} n^{-\delta}\log^{-1}(n).
\end{split}
\end{equation}

With an eye on the upcoming union bound argument over all useful and nice vertex sets~$W$, 
we first fix an arbitrary set~$W \subseteq V$ containing at least~$|W| \ge \ell_1(W)$ many vertices 
(to clarify: here we do not yet assume that~$W$ is useful and nice). 
For technical reasons, in the following we will often 
 (sometimes tacitly) condition on the random~variable
\begin{equation*}
\Xi_W \: := \: \bigpar{\indic{uv\in E(\Gnp)}}_{uv \in \binom{V}{2} \setminus \binom{W}{2}}, 
\end{equation*}
which encodes the edge-status of all vertex pairs of~$\Gnp$ except for those inside~$W$. 
Note that~$\Xi_W$ contains enough information to determine whether~$W$ is useful and also whether~$W$ is nice.
Furthermore, when~$W$ is useful and nice, then~$\Xi_W$ also contains enough information to determine the collection~$\cc=\cc(W)$ of at least
\begin{equation*}
|\cc| \geq \portion|\acc| = \portion \binom{a}{k-m}b^{m}
\end{equation*}
many inclusion-maximal clique candidates with respect to~$W$, that is guaranteed by \refL{lem:counting}. 
Denoting by~$X_W$ the number of~$K \in \cc$ that are cliques in~$\Gnp$, 
using~$0 \le X_W \le Y_W$ it follows that
\begin{equation}\label{eq:YW:bound:0}
\begin{split}
\Pr(\text{$Y_W=0$, and $W$ is useful and nice}) 
& = \E\bigpar{\Pr(Y_W=0 \mid \Xi_W)\indic{\text{$W$ is useful and nice}}} \\
& \le \E\bigpar{\Pr(X_W=0 \mid \Xi_W)\indic{\text{$W$ is useful and nice}}}.
\end{split}
\end{equation}
As usual, for useful and nice~$W$, we now write~$X_W$ as a sum of indicator random variables:  
\begin{equation*}
X_W := \sum_{K \in \cc}I_K 
\qquad 
\text{ with } 
\qquad 
I_K := \indic{\text{$K$ is clique in $G_{n,p}$}}.
\end{equation*}
After conditioning on~$\Xi_W$, note that all potential edges inside~$W$ are still included independently with probability~$p$. 
By applying Janson's inequality (see, e.g., \mbox{\cite[Theorem~2.14]{JLR}})
to~$X_W$,
for useful and nice~$W$ it thus follows that 
\begin{equation}\label{eq:XW:bound}
\Pr(X_W=0 \mid \Xi_W) \le  \exp\biggpar{-\frac{\mu^2}{2(\mu+\Delta)}} ,
\end{equation}
where the associated expectation and correlation parameters~$\mu$ and~$\Delta$ appearing in~\eqref{eq:XW:bound} are
\begin{equation*}
\mu := \sum_{K \in \cc}\E(I_K \mid \Xi_W)
\qquad 
\text{ and } 
\qquad 
\Delta := \sum_{K \in \cc}\sum_{\substack{J \in \cc: J \neq K,\\|J \cap K| \ge 2}}\E(I_KI_J \mid \Xi_W).
\end{equation*}
Recalling that $\cc \subseteq \acc$ (see \refL{lem:counting}), by~\eqref{eq:almost-cc} all clique candidates~$K \in \cc$ satisfy~$|K|=k$, $|K \cap A|=k-m$ and $|K \cap B_i|=1$ for all~$i \in [m]$. 
Since all all potential edges inside~$W$ are included independently with probability~$p$ (even after conditioning on~$\Xi_W$), 
for useful and nice~$W$ standard random graph estimates give 
\begin{equation*}
\begin{split}
\mu 
& =  |\cC| \cdot p^{\binom{k}{2}} 
\ge \nu \binom{a}{k-m}b^{m} p^{\binom{k}{2}} 
\end{split}
\end{equation*}
and (by taking all possible overlaps of the clique candidates~$K,J \in \cc$ in~$A$ and~$B_1, \ldots, B_m$ into account) also 
\begin{equation}\label{eq:XW:Delta}
\begin{split}
\mu + \Delta  
& = \sum_{K \in \cc}\sum_{\substack{J \in \cc:\\ |J \cap K| \ge 2}}\E(I_K I_J \mid \Xi_W)
 = \sum_{2 \le r \le k}\sum_{K \in \cc}\sum_{\substack{J \in \cc:\\ |J \cap K| =r}}p^{2\binom{k}{2}-\binom{r}{2}}\\
& \le \sum_{2 \le r \le k} \sum_{\substack{x+y=r:\\ 0 \le x \le k-m\\0 \le y \le m}} \underbrace{|\cC| \cdot \binom{k-m}{x}\binom{a-(k-m)}{(k-m)-x}\binom{m}{y}b^{m-y} \cdot p^{2\binom{k}{2}-\binom{r}{2}}}_{=: \Delta_{x,y}} .
\end{split}
\end{equation}
Fix~$x+y=r$. 
Note that in $\Delta_{x,y}$ we have~$\binom{a-(k-m)}{(k-m)-x} \le \binom{a}{(k-m)-x}$ and~$\binom{k-m}{x}\binom{m}{y} \le \binom{k}{x+y}\le k^{x+y}=k^r$. 
Recalling that~$a \ge |W|/4 \ge \ell_0/4$, from~\eqref{ass:partition} we know that~$a \geq 2k^2 \ge 2rk$, 
so using~$a \ge b$ and~$x \le x+y= r$ it follows that $\bigpar{1-(k-m)/a}^x \ge 1-rk/a \ge 1/2$ and 
\begin{equation}\label{eq:XW:ratio}
\begin{split}
\frac{\mu^2}{\Delta_{x,y}} & \ge \frac{\nu \binom{a}{k-m} b^yp^{\binom{r}{2}}}{\binom{a}{(k-m)-x} k^r}\\
& = \nu \cdot \frac{(a-(k-m)+x)!}{(a-(k-m))! a^x} \cdot \frac{(k-m-x)!}{(k-m)!k^r}\cdot a^x b^y p^{\binom{r}{2}}\\
& \ge \nu \cdot \bigpar{1-(k-m)/a}^x \cdot k^{-(x+r)} \cdot b^{x+y}p^{\binom{r}{2}}\\
& \ge \nu/2 \cdot (b/k^2)^rp^{\binom{r}{2}} =: \lambda_r.
\end{split}
\end{equation}
Note that the ratio~$\lambda_{r+1}/\lambda_{r} = b/k^2 \cdot p^{r}$ is decreasing in~$r$, which implies that the minimum of~$\lambda_2, \ldots, \lambda_k$ is either~$\lambda_2$ or~$\lambda_k$, i.e., that  ${\min_{2 \le r \le k} \lambda_r} = {\min\{\lambda_2,\lambda_k\}}$. 
Using~$\Delta_{x,y} \le \mu^2/\lambda_r$ and the definition of~$\lambda_r$ 
together with the observation that in~\eqref{eq:XW:Delta} there are at most~$k$ choices of~$x,y$ with~$x+y=r$ for a given~$r$ (due to $0 \le x \le k-m \le k-1$), 
for useful and nice~$W$ it follows that the exponent in inequality~\eqref{eq:XW:bound} satisfies
\begin{equation}\label{eq:Janson:exponent}
\frac{\mu^2}{2(\mu+\Delta)} 
\ge \frac{\mu^2}{2\sum_{2 \le r \le k} k \mu^2/\lambda_r}
\ge \frac{\min_{2 \le r \le k} \lambda_r}{2 k^2}  
= 
\frac{\nu}{4} \cdot \min\left\{\frac{(b/k^2)^2p}{k^2} , \; \frac{(b/k^2)^k p^{\binom{k}{2}}}{k^2}\right\}.
\end{equation}
Recall that~$p=n^{-\rho}$ and~$b = \ceil{\ell_1(W)/(4m)}$. 
Using first the assumption~\eqref{ass:Janson:k:bound} and estimates~\eqref{eq:l1W:lower}--\eqref{eq:l1W:ratio} for~$m < k \le \log(n)$ and~$\ell_1(W)$, 
and finally the estimate $\delta \leq 1/2 - \rho - 9 \log(\log(n))/\log(n)$ from~\eqref{ass:Janson-delta-bound}, it follows~that 
\begin{align}\label{eq:Janson:exponent:1}
\frac{(b/k^2)^2p}{k^2|W| \log n}
\geq \frac{\ell_1(W)p}{16 k^6 m^2 \log n} \cdot  \frac{\ell_1(W)}{|W|} \gg \frac{n^{1-2\delta}p^2}{\log^{12}(n)} 
= \frac{n^{1-2\delta -2\rho}}{\log^{12}(n)} \gg 1.
\end{align}
By analogous reasoning,
using the estimate $\delta \le \frac{k-1}{k}\bigpar{1-\rho(k/2+1)}- 9 \log(\log(n))/\log(n)$ from~\eqref{ass:Janson-delta-bound}, 
and that~${4+7/(k-1)} \le 7.5$ follows from~\eqref{ass:Janson:k:bound}, we infer~that 
\begin{equation}\label{eq:Janson:exponent:2}
\begin{split}
\frac{(b/k^2)^k p^{\binom{k}{2}}}{k^2|W| \log n} & \ge \frac{1}{4k^4 m \log n} \cdot \frac{\ell_1(W)}{|W|} \cdot \biggpar{ \frac{\ell_1(W)p^{k/2}}{4mk^2}}^{k-1}\\
& \ge \frac{n^{-\delta}}{32\log^{7}(n)} \cdot \biggpar{ \frac{n^{1-\delta}p^{k/2+1}}{32\log^{4}(n)}}^{k-1}\\
& \gg \Biggpar{ \frac{n^{1-\tfrac{k}{k-1}\delta-\rho(k/2+1)}}{\log^{8}(n)}}^{k-1} \gg 1 .
\end{split}
\end{equation}
Since~$\Xi_W$ determines whether~$W$ is useful and nice, by inserting~\eqref{eq:Janson:exponent}--\eqref{eq:Janson:exponent:2} into~\eqref{eq:XW:bound}  
it follows that, say,  
\begin{equation}\label{eq:YW:bound:2}
\Pr(X_W=0 \mid \Xi_W)\indic{\text{$W$ useful and nice}} \ll  \exp\Bigpar{-2|W| \log n} = n^{-2|W|} .
\end{equation}

Finally, by inserting the estimate~\eqref{eq:YW:bound:2} into~\eqref{eq:YW:bound:0}, the probability that~$Y_W=0$ holds for some useful and nice vertex set~$W \subseteq V$ is thus routinely seen (by a standard union bound argument that takes all possible~$W \subseteq V$ into account) to be at most
\begin{equation*}
\begin{split}
\sum_{W \subseteq V: |W| \ge \ell_1(W)} \Pr( \text{$Y_W=0$, and $W$ useful and nice}) 
\le \sum_{W \subseteq V: |W| \ge 1} n^{-2|W|} 
\le \sum_{w \ge 1} n^w \cdot n^{-2w} = o(1) .
\end{split}
\end{equation*}
This completes the proof of \refL{lem:Janson}, 
since by \refL{lem:density} whp every useful set~$W$ is also nice.
\end{proof}

\begin{remark}[Improved estimate for~$m=1$]\label{rem:Janson}
For later reference (see \refS{sec:main:rem}) we record that, in the special case~$m=1$, 
in the above application of Janson's inequality~\eqref{eq:XW:bound} we can improve the exponent estimate~\eqref{eq:Janson:exponent}~to
\begin{equation}\label{eq:Janson:exponent:m1}
\frac{\mu^2}{2(\mu+\Delta)} 
\ge 
\frac{\nu}{4} \cdot \min\left\{\frac{(a/k^2)\cdot (b/k^2)p}{k^2} , \; \frac{(a/k^2) \cdot (b/k^2)^{k-1} p^{\binom{k}{2}}}{k^2}\right\}.
\end{equation}
Indeed, the key observation is that in~\eqref{eq:XW:Delta} we have~$0 \le y \le m = 1$ and~$r \ge 2$, 
so that in the subsequent estimates with fixed~$x+y=r$ we have~$x=r-y \ge 1$. 
This implies that estimate~\eqref{eq:XW:ratio} now holds with~$\lambda_r := \nu/2 \cdot (a/k^2) \cdot (b/k^2)^{r-1}p^{\binom{r}{2}}$, 
and by mimicking the argument leading to~\eqref{eq:Janson:exponent} we then readily infer~\eqref{eq:Janson:exponent:m1}.
\end{remark}

\subsection{Technical part: deferred choice of parameters}\label{sec:choices}
In this section we complete the proof of \refT{thm:max-clique} (modulo the routine proofs of \refL{lem:non-nbr:lb} and~\ref{lem:density} in \refS{sec:deferred}) 
by giving the deferred proof of \refL{lem:choices}, i.e., we show that we can always choose suitable parameters~$\delta$, $m$ and~$k$ satisfying the technical inequalities~\eqref{ass:Janson:k:bound}--\eqref{eq:m1:delta}. 
Our concrete choices in \refS{sec:param:choices} are in some sense already a significant part of the proof, as checking that they indeed satisfy the desired inequalities is tedious, but conceptually routine. 
Here one difficulty is that for~$p \ge n^{-o(1)}$ and~$p \le n^{-o(1)}$ we need to establish rather different guarantees on~$\delta$, which are required for the different conclusions of Theorem~\ref{thm:main1} and~\ref{thm:main2}.
However, the major technical difficulty is that for very small edge-probabilities~$p=p(n)$ there is hardly any elbow room in the arguments, and so for~$n^{-2/5+\eps} \le p \ll n^{-1/3}$ and $n^{-1/3+\eps} \le p < n^{-1/4-\eps}$ we effectively need to do some calculations for~$k=3$ and~$k=4$ on an ad-hoc basis (rather than relying on general formulas and estimates). 

Recall that \refT{thm:main1} assumes that $p=p(n)$ satisfies $n^{-2/5 + \eps} \ll p \ll n^{-1/3}$ or $n^{-1/3 + \eps} \ll p \ll 1$, where~$\eps >0$ is fixed. 
Furthermore, we have~$\sigma=1/100$, $\alpha=4/5$, $\portion=1/10$ and~$\rho = \log_n(1/p)$.

\subsubsection{Concrete choices of~$\delta$, $m$ and~$k$}\label{sec:param:choices}
In this subsection we define the parameters $m$, $k$ and $\delta$ by a case distinction.  

\textbf{Case $n^{-\sigma} \le p \ll 1$: }
Here we 
define~$\delta=\delta(\rho,n)$ as in~\eqref{eq:m2:delta}, 
and set 
\begin{equation}
\label{eq:m2:km}
\begin{split}
m := \biggfloor{\frac{2}{3\rho}}
\qquad \text{and} \qquad 
k := \biggceil{\frac{1}{\rho} + \frac{1}{2}} ,
\end{split}
\end{equation}
for which it is routine (using $0 < \rho \le \sigma=1/100$)  to check that~$m \ge 2/(3\rho)-1 \ge 2$ and, say, 
\begin{equation}\label{eq:m2:inequality}
(k-m-1)\rho \ge 1/4.
\end{equation}

\textbf{Case $n^{-2/5+\eps} \le p \ll n^{-1/3}$ or $n^{-1/3+\eps} \le p < n^{-\sigma}$: }
Here we
define~$\delta=\delta(\eps,\sigma)$ as in~\eqref{eq:m1:delta}, 
and set 
\begin{equation}
\label{eq:m1:km}
\begin{split}
m  := 1
\qquad \text{and} \qquad 
k := \biggceil{\frac{1}{\rho} + \frac{\indic{\rho \le 4/15}}{2}} , 
\end{split}
\end{equation}
for which a basic case distinction (depending on whether $\sigma \le \rho \le 1/3$ or $1/3 \le \rho \le 2/5$, in which case either~$k \ge 1/\rho$ or~$k=3$) shows that, say, 
\begin{equation}\label{eq:m1:inequality}
(k-m-1)\rho \ge 1/4.
\end{equation}

\subsubsection{Inequality~\eqref{ass:Janson:k:bound}: bounds on~$m$ and~$k$}
In this subsection we verify inequality~\eqref{ass:Janson:k:bound} of~\refL{lem:choices}, using that~$\rho = \log_n(1/p) \le 2/5$ and~$p \ll 1$. 

We start by noting that~$k \le 1/\rho + 3/2 \le 2/\rho = 2\log_{1/p}(n) \ll \log(n)$. 
Furthermore, since~$m$ and~$k$ are both integers,  estimates~\eqref{eq:m2:inequality} and~\eqref{eq:m1:inequality} each imply~$k - m = \ceil{k - m} \ge \ceil{1+1/(4 \rho)} \ge 2$, establishing inequality~\eqref{ass:Janson:k:bound}.

\subsubsection{Inequality~\eqref{ass:Janson-delta-bound}: bounds on~$\delta$}\label{sec:technical:delta}
In this subsection we verify inequality~\eqref{ass:Janson-delta-bound} of \refL{lem:choices} by several case distinctions, using that~$\sigma=1/100$. 

We start with the case when~$\rho \leq \sigma$. Then $1/\rho + 1/2 \le k \le 1/\rho + 3/2$ and thus 
\begin{equation}\label{eq:test0}
\min\biggcpar{\frac{1}{2} - \rho, \ \frac{k-1}{k}\Bigpar{1-\rho(k/2+1)}} 
\ge 
\min\biggcpar{\frac{1}{2} - \rho, \ 1-(\rho k/2+\rho)-1/k} 
\ge \frac{1}{2} - 3\rho .
\end{equation}
We next consider the case when $\sigma < \rho < 4/15$. Then $1/\rho + 1/2 \le k \le 1/\rho + 3/2$ and thus 
\begin{equation}\label{eq:test1}
\min\biggcpar{\frac{1}{2} - \rho, \ \frac{k-1}{k}\Bigpar{1-\rho(k/2+1)}} \geq \min\biggcpar{\frac{1}{2} - \rho, \ (1-\rho)\bigpar{1/2-7\rho/4}} 
\ge \sigma.
\end{equation}
Thereafter we consider the case when $4/15 \leq \rho \leq 1/3 - \eps$. Then~$k = \ceil{1/\rho}=4$ and thus 
\begin{equation}\label{eq:test2}
\min\biggcpar{\frac{1}{2} - \rho, \ \frac{k-1}{k}\Bigpar{1-\rho(k/2+1)}} 
\ge 
\min\biggcpar{\sigma, \ \frac{3}{4}\bigpar{1 - 3\rho}} \ge \min\bigcpar{\sigma, \: 9\eps/4} .
\end{equation}
Finally we consider the remaining case when $1/3 \le \rho \le 2/5 -\eps$. Then~$k = \ceil{1/\rho}= 3$ and thus 
\begin{equation}\label{eq:test3}
\min\biggcpar{\frac{1}{2} - \rho, \ \frac{k-1}{k}\Bigpar{1-\rho(k/2+1)}} 
\ge 
\min\biggcpar{\sigma, \ \frac{2}{3}\bigpar{1 - 5 \rho/2}} 
\ge \min\bigcpar{\sigma, \: 5\eps/3} .
\end{equation}
To sum up: inspecting our choices~\eqref{eq:m2:delta} and~\eqref{eq:m1:delta} of~$\delta$ for~$p \ge n^{-\sigma}$ and~$p < n^{-\sigma}$, 
a moment's thought reveals that the estimates~\eqref{eq:test0}--\eqref{eq:test3} together establish inequality~\eqref{ass:Janson-delta-bound} for all sufficiently large~$n$ (depending on~$\eps$).

\subsubsection{Inequality~\eqref{ass:s}: bounds on~$s$ and~$m$}\label{sec:ass:s}
In this subsection we verify inequality~\eqref{ass:s} of \refL{lem:choices}. 
Recall that~$s =\floor{\delta \log_{1/(1-p)}(n)}$ and~$\rho = \log_n(1/p)$ by~\eqref{def:s} and~\eqref{eq:param}, 
and that~$\delta \ge \min\{\eps,1/200\} = \Omega(1)$ by the discussion below \refL{lem:choices}. 
Furthermore, as in~\eqref{eq:chi:lower}, using~$p \ll 1$ we have~$\log_{1/(1-p)}(n) = (1-o(1))\log(n)/p$.  
Taking the two different definitions~\eqref{eq:m2:km} and~\eqref{eq:m1:km} of~$m$ into account, 
using~$p \ll 1$ it therefore follows~that
\begin{equation}\label{eq:ass:s}
\frac{s+1}{m} 
\ge \frac{\delta \log_{1/(1-p)}(n)}{\max\{1, 1/\rho\}}
\ge (1-o(1)) \frac{\delta \: \min\bigcpar{\log(n), \: \log(1/p)}}{p} 
\gg 1 ,
\end{equation}
which in view of~$m \ge 1$ readily establishes inequality~\eqref{ass:s}, with room to spare.

\subsubsection{Inequality~\eqref{ass:partition}: bound on $\ell_0$ and $n^{1-\delta}p$}\label{sec:aux:ell0_np}
In this subsection we verify inequality~\eqref{ass:partition} of \refL{lem:choices}. 
Recall that $\ell_0=\ell_1(\emptyset)$ by~\eqref{def:ell0}. 
By proceeding as in~\eqref{eq:l1W:lower}, using~$p=n^{-\rho}$ and the estimate~$\delta \leq 1/2 - \rho - 9 \log(\log(n))/\log(n)$ from~\eqref{ass:Janson-delta-bound} it follows that 
\begin{equation}\label{eq:l1W:lower_w_choice}
\min\Bigcpar{\ell_0, \: n^{1-\delta}p} \gg n^{1-\delta}p \log^{-2}(n) = \frac{n^{1-\delta-\rho}}{\log^2(n)} \gg n^{1/2} .
\end{equation}
This readily establishes inequality~\eqref{ass:partition}, since~\eqref{ass:Janson:k:bound} implies $m < k \le \log(n)$.

\subsubsection{Inequalities~\eqref{eq:density:1} and~\eqref{eq:density:2}: bounds involving~$x_i$ and~$m$}
In this subsection we verify  inequalities~\eqref{eq:density:1} and~\eqref{eq:density:2}  of \refL{lem:choices}.

With an eye on inequality~\eqref{eq:density:1}, 
using the definition~\eqref{eq:param} of~$r_i=e^{\gr i}$ and~$\gr = \log^{-4}(n)$ we see that 
\begin{equation}\label{eq:riminus1:lower}
r_i-1 \ge r_1 - 1 = e^{\gr}-1 
= (1+o(1)) \log^{-4}(n) \qquad \text{ for~$i \ge 1$.}
\end{equation}
It is well-known (see e.g., the proof of Corollary~8 in~\cite{SW2020}) that $\phi(x) = (1+x)\log(1+x) - x$ satisfies
\begin{equation}\label{eq:phi-estimate:1}
\phi(x) \ge \min\{x,x^2\}/2 \qquad \text{ for~$x \ge 0$.}
\end{equation}
Using these estimates together with~\eqref{eq:l1W:lower_w_choice} and~$p \ge n^{-2/5+\eps}$ it follows that $\phi(r_i - 1) \gg \log^{-9}(n)$ and 
\begin{equation}\label{eq:density1:auxiliary}
\frac{\log^3(n)/\ell_0}{\phi(r_i - 1)p} 
\ll \frac{ \log^{12}(n)}{p n^{1/2}} \ll 1 . 
\end{equation}
Combining these estimates with~\eqref{eq:l1W:lower_w_choice} and
the definition~\eqref{eq:param} of $x_i = \log(n)/[\phi(r_i - 1)p]$, it  follows~that
\begin{equation}\label{eq:density1:lower}
\begin{split}
\frac{\ceil{x_i}\bigpar{\phi(r_i-1)p - \log^3(n)/\ell_0}}{1+\max\bigcpar{\log\bigpar{n\log^2(n)e/\ell_0},0}}
& \ge 
\frac{(1-o(1)) x_i\phi(r_i-1)p}{\log\bigpar{n^{1/2+o(1)}}} 
\ge 2-o(1) > 1 ,
\end{split}
\end{equation}
which establishes inequality~\eqref{eq:density:1}. 

We now turn to inequality~\eqref{eq:density:2}, for which it suffices to consider the case~$m \ge 2$, 
where we actually have ${m =\floor{2/(3\rho)}}$ by~\eqref{eq:m2:km}. Here we use the basic inequality
\begin{equation}\label{eq:phi-estimate:2}
\phi(x) \ge \phi(x)-1 = (1+x)\log((1+x)/e)
\end{equation}
together with~\eqref{eq:l1W:lower_w_choice} as well as~$\alpha=\Theta(1)$ and~$p \ll 1$ to deduce that $\phi(\alpha/p-1)p \ge (1-o(1)) \alpha \log(1/p)$ and 
\[
\frac{\log^3(n)/\ell_0}{\phi(\alpha/p-1)p} \ll \frac{ \log^3(n)}{\log(1/p) n^{1/2}} \ll 1 .
\]
Combining these estimates with~\eqref{eq:l1W:lower_w_choice} and~$\log(1/p) = \rho \log n$, 
using~$\alpha=4/5$ and~$p \ll 1$ together with $\phi(\alpha/p-1)p \ge (1-o(1)) \alpha \log(1/p)$ 
 and ${m+1}\ge 2/(3\rho)$  it then follows~that
\begin{align*}
\frac{(m+1) \bigpar{\phi(\alpha/p-1)p - \log^3(n)/\ell_0}}{1+\max\bigcpar{\log\bigpar{n\log^2(n)e/\ell_0},0}}
&  \ge \frac{(1-o(1))(m+1)\phi(\alpha/p-1)p }{\log\bigpar{n^{1/2+o(1)}}} \ge  \frac{(2-o(1))(m+1)\alpha \log(1/p)}{\log(n)} \\
 & \ge (8/5-o(1)) (m+1)\rho \ge 16/15-o(1) > 1 ,
\end{align*}
which establishes inequality~\eqref{eq:density:2}.

\subsubsection{Inequality~\eqref{ass:counting}: bound on~$\Lambda$ from \refL{lem:counting}}\label{sec:Lambda}
In this subsection we verify the remaining inequality~\eqref{ass:counting} of \refL{lem:choices}. 
Recalling~\eqref{eq:def:Lambda}, our goal is to show~that
\begin{equation}\label{eq:def:Lambda:recall}
\Lambda = \underbrace{(m+1)x_1(r_2 p)^{k-m}  + n (r_1 p)^k}_{=: \Lambda_0} + \indic{m \ge 2} \Pi_{\alpha} + \indic{m=1}\Bigsqpar{\Pi_{1/\log(n)} + 0.7}
\end{equation}
is at most~$1-\portion=0.9$, 
where~$\Pi_{\alpha}$ and~$\Pi_{1/\log(n)}$ are defined as in~\eqref{eq:accX:1}. 

We start with~$\Lambda_0$ as defined in~\eqref{eq:def:Lambda:recall}. 
Recall that~\eqref{ass:Janson:k:bound} implies  $m < k \le \log(n)$.
Using~\eqref{eq:riminus1:lower}--\eqref{eq:phi-estimate:1} and~$r_i = e^{i \gr}$ together with~$p=n^{-\rho}$ and~$\gr = \log^{-4}(n)$, in view of the auxiliary estimates~\eqref{eq:m2:inequality} and~\eqref{eq:m1:inequality} as well as the definitions~\eqref{eq:m2:km} and~\eqref{eq:m1:km} of~$k$ it follows that 
\begin{equation}\label{eq:def:Lambda0:1}
\begin{split}
\Lambda_0 & \le \log(n) \cdot \frac{\log(n)}{\phi(r_1-1)} \cdot e^{2 \gr (k-m)} p^{k-m-1}  + e^{\gr k} n p^k  \\
& \le (2+o(1)) \cdot \log^{10}(n) \cdot n^{-(k-m-1)\rho} + 2 n^{1-\rho k} \\
& \le \underbrace{3 \cdot \log^{10}(n) \cdot n^{-1/4} + 2 p^{1/2}\indic{\rho \le 4/15}}_{=o(1)} + 2 n^{-(\rho\ceil{1/\rho}-1)} \indic{\rho > 4/15} .
\end{split}
\end{equation}
Recall that~$\rho = \log_n(1/p)$. 
By distinguishing whether~$n^{-2/5} \ll p \ll n^{-1/3}$ and~$n^{-1/3} \ll p \le n^{-4/15}$ (in which case either~$\ceil{1/\rho}=3$ and~$\ceil{1/\rho}=4$), 
it is not difficult (by writing~$p \le n^{-1/3}/\omega$ and $p \le n^{-4/15}$, say)  
to verify that~$\rho\ceil{1/\rho}-1 \gg 1/\log(n)$ when~$\rho > 4/15$. 
Inserting this estimate into~\eqref{eq:def:Lambda0:1} then~yields
\begin{equation}\label{eq:def:Lambda0:2}
\Lambda_0 = o(1)  .
\end{equation}

We now turn to~$\Pi_{\alpha}$ and~$\Pi_{1/\log(n)}$ appearing in~\eqref{eq:def:Lambda:recall}. 
By the definitions in~\eqref{eq:param} and~\eqref{eq:accX:1} we~have
\begin{equation}\label{eq:Pi:alpha}
\Pi_\alpha = \log(n) \sum_{i \geq 2:pr_{i} \leq \alpha} 
\frac{(pr_{i})^{k-m}}{\phi(r_i-1)p} \Bigpar{e^{\gr (k-m)}-1} .
\end{equation}
The sum in~\eqref{eq:Pi:alpha} is intuitively dominated (up to constant factors) by the last term, and to make this rigorous we shall now separately analyze the behavior of the term~$\phi(r_i-1)p$ and also the behavior of the sum without the term~$\phi(r_i-1)p$. 
Turning to the details, 
using the definition~\eqref{eq:param} of~$r_i=e^{\gr i}$ and~$\gr = \log^{-4}(n)$ together with~$k \le \log n$, 
for any~$\beta > 0$ and~$1 \le s \le k$ it follows that $\gr s=o(1)$ as well as~$e^{-\gr s} = 1- \gr s(1+o(1))$ and 
\begin{equation}\label{eq:Pi:alpha:sum}
\sum_{i \geq 0:pr_{i} \leq \beta} 
(pr_{i})^{s} \le \beta^{s} \sum_{j \ge 0} e^{-\gr s j} \le \frac{\beta^{s}}{1-e^{-\gr s}}  = \frac{(1+o(1)) \beta^{s}}{\gr s}.
\end{equation}
Next we carefully estimate the term~$\phi(r_i-1)$ in~\eqref{eq:Pi:alpha}.
Using estimates~\eqref{eq:riminus1:lower}--\eqref{eq:phi-estimate:1}, for all~$i \ge 1$ we have 
\begin{equation}\label{eq:Pi:phi:small}
\phi(r_i-1) \ge \min\bigcpar{(r_i-1), \: (r_i-1)^2}/2  \ge (\tfrac{1}{2}+o(1)) \log^{-8}(n) \gg \log^{-9}(n).
\end{equation}
Using estimate~\eqref{eq:phi-estimate:2}, for~$i \ge 1$ we also have
\begin{equation}\label{eq:Pi:phi:large}
\phi(r_i-1) \ge r_i \log(r_i/e) \ge 
\begin{cases}
r_i \quad & \text{ if $r_i \ge e^2$}, \\
r_i \log\bigpar{p^{\sigma - 1}} \quad & \text{ if $r_i \ge ep^{\sigma - 1}$}.
\end{cases}
\end{equation}
We are now ready to estimate $\Pi_{\alpha}$ from~\eqref{eq:Pi:alpha}:  
by distinguishing whether~$r_i \le e^2$ or~$r_i \ge e^2$ or~$r_i \ge ep^{\sigma-1}$, 
using estimates~\eqref{eq:Pi:alpha:sum}--\eqref{eq:Pi:phi:large} 
together with $k-m-1 \ge 1$ and $e^{\gr (k-m)}-1 = \gr (k-m)(1+o(1))$ 
it follows that, say, 
\begin{equation}\label{eq:Pi:alpha:2}
\begin{split}
\Pi_\alpha &\le 
\log(n) \cdot \Biggsqpar{
\sum_{i \geq 2:pr_{i} \leq e^2p} 
\frac{(pr_{i})^{k-m}}{\log^{-9}(n)}
+ 
\sum_{i \geq 2:p r_i \le ep^{\sigma}} 
(pr_{i})^{k-m-1}
+ 
\sum_{i \geq 2:pr_{i} \leq \alpha} 
\frac{(pr_{i})^{k-m-1}}{\log(p^{\sigma - 1})}
}
\cdot 
\bigpar{e^{\gr (k-m)}-1} \\
& \le (1+o(1)) \cdot \biggsqpar{ \log^{10}(n) \cdot (e^2p)^{k-m} + \log(n) \cdot (ep^{\sigma})^{k-m-1} + \frac{\log (n)}{\log(p^{\sigma - 1})} \alpha^{k-m-1}} \cdot \frac{k-m}{k-m-1}.
\end{split}
\end{equation}
Using~$e^2 p  \ll p^{1/2}$ and~$ep^{\sigma} \ll p^{\sigma/2}$ as well as $p=n^{-\rho}$ together with  estimates~\eqref{eq:m2:inequality} and~\eqref{eq:m1:inequality},  we see~that 
\begin{equation}\label{eq:Pi:alpha:first}
(e^2p)^{k-m} + (ep^{\sigma})^{k-m-1} \ll n^{-(k-m)\rho/2}  + n^{-\sigma(k-m-1)\rho/2}
\le n^{-1/8} + n^{-\sigma/8} \ll \log^{-10}(n). 
\end{equation}
To sum up, by combining~\eqref{eq:Pi:alpha:2}--\eqref{eq:Pi:alpha:first} with~$k-m-1 \ge 1$ it follows that $\tfrac{k-m}{k-m-1} \le 2$~and 
\begin{equation}\label{eq:Pi:alpha:3}
\begin{split}
\Pi_\alpha & \le o(1) + (2+o(1)) \cdot \frac{\alpha^{k-m-1}\log(n)}{(1-\sigma)\log(1/p)} .
\end{split}
\end{equation}

After these preparations, we are now ready to bound~$\Lambda$ from~\eqref{eq:def:Lambda:recall} via a case distinction. 
First we consider the case~$m \ge 2$, where~$n^{-\sigma} \le p \ll 1$. 
Using~$\alpha = 4/5$ and~$\rho = \log_n(1/p)$ together with the auxiliary estimates~\eqref{eq:m2:inequality} and~\eqref{eq:m1:inequality}, 
in view of~$\rho \le \sigma=1/100$ it follows via basic calculus~that 
\begin{equation}\label{eq:Pi:alpha:second}
\frac{\alpha^{k-m-1}\log(n)}{(1-\sigma)\log(1/p)} \le \frac{(4/5)^{1/(4\rho)}}{0.99 \rho}
\le \frac{(4/5)^{1/(4\sigma)}}{0.99 \sigma} < 0.4 ,
\end{equation}
which combined with~\eqref{eq:def:Lambda:recall}, \eqref{eq:def:Lambda0:2} and~\eqref{eq:Pi:alpha:3} then yields
\begin{equation}\label{eq:Pi:alpha:Lambda}
\Lambda \le \Lambda_0 + \Pi_{\alpha} \le o(1) + (2+o(1)) \cdot 
 0.4 < 0.9 = 1-\portion . 
\end{equation}

Finally we consider the remaining case~$m=1$, where~$p \le n^{-\sigma}$. 
Note that the arguments leading to~\eqref{eq:Pi:alpha:3} only used~$\alpha > 0$, i.e., remain valid (for all sufficiently large~$n$) when we replace~$\alpha$ with~$1/\log(n)$. 
Using~$1/p \ge n^{\sigma}$ and~$\sigma = 1/100$ together with~$k-m-1 \ge 1$, 
in view of~\eqref{eq:Pi:alpha:3}  it therefore follows~that 
\begin{equation*}
\Pi_{1/\log(n)} \le o(1) + O(1)  \cdot \biggpar{\frac{1}{\log(n)}}^{k-m-1} \cdot \frac{\log(n)}{\log(1/p)} \le o(1) + \frac{O(1)}{\log(n)} = o(1), 
\end{equation*}
which combined with~\eqref{eq:def:Lambda:recall} and \eqref{eq:def:Lambda0:2} then yields
\begin{equation}\label{eq:Pi:alpha:Lambda:2}
\Lambda \le \Lambda_0 + \Pi_{1/\log(n)} + 0.7 \le o(1) + o(1) + 0.7< 0.9 = 1-\portion.
\end{equation}
Note that estimates~\eqref{eq:Pi:alpha:Lambda} and~\eqref{eq:Pi:alpha:Lambda:2} together establish inequality~\eqref{ass:counting}, completing the proof of~\refL{lem:choices}.
\qed

\begin{remark}[Relaxing the assumption~$p \ll n^{-1/3}$]\label{rem:np13:improvement}
In the above proof of \refT{thm:max-clique} (and the upcoming proof of \refT{rem:max-clique}) the assumed bound $p \ll n^{-1/3}$ can easily be relaxed to~$p \le \tfrac{1}{3}n^{-1/3}$, say. 
Indeed, $p \ll n^{-1/3}$ is only used in the proof of \refL{lem:choices} to conclude that~$e^{\gr k} n p^k=o(1)$ in the arguments~\eqref{eq:def:Lambda0:1}--\eqref{eq:def:Lambda0:2} leading to the estimate~$\Lambda_0=o(1)$.  
The crux is that for~$p \le \tfrac{1}{3}n^{-1/3}$ we have~$e^{\gr k} n p^k \le (1+o(1)) \cdot 3^{-3} < 0.04$ and thus~$\Lambda_0 \le 0.05$, say, 
which in turn readily ensures 
that the arguments in~\eqref{eq:Pi:alpha:Lambda} and~\eqref{eq:Pi:alpha:Lambda:2} still give the key conclusion~$\Lambda < 1-\portion$. 
It follows that we can relax the assumption~$p \ll n^{-1/3}$ in \refT{thm:main1} and~\ref{thm:main3} to~$p \le \tfrac{1}{3}n^{-1/3}$.  
\end{remark}

\section{Refinement of \refS{sec:main}: proof of \refT{rem:max-clique}}\label{sec:main:rem}
In this section we prove \refT{rem:max-clique}, which is a refinement of our main technical result \refT{thm:max-clique}. 
Our approach is to modify the proof from \refS{sec:main} when the edge-probability~$p=n^{-2/5+\eps}$ satisfies~$(\log n)^\omega n^{-2/5} \le p \ll n^{-1/3}$. 
Following the definition~\eqref{eq:m1:km} from \refS{sec:param:choices}, in this case we have~$m=1$ and~$k=3$.

Ignoring technicalities, the main proof idea is that we can use \refR{rem:Janson} to improve the constraint~\eqref{ass:Janson-delta-bound} on~$\delta=\delta(n,p,\eps)$, 
which is mainly used to lower bound the exponents~\eqref{eq:Janson:exponent:1}--\eqref{eq:Janson:exponent:2} of Janson's inequality.
As we shall see in \refS{sec:main:rem:janson} below, this intuitively will allow us to improve the constraint~$\tfrac{k-1}{k}(1-\rho(k/2+1))$ in~\eqref{ass:Janson-delta-bound} to $1-\rho(k/2+1) = 1-5 \rho/2=5 \eps/2$, which matches (up to second order error terms) the desired form of~$\delta$.

The details of the proof of \refT{rem:max-clique} are spread across the remainder of this section. 
Recall that we are in the case~$m=1$, and so in~\eqref{eq:m1:delta} of \refS{sec:main} we previously set~$\delta = \min \{\eps,\sigma/2\}$. 
In view of~$\eps \gg \log(\log(n))/\log(n)$, 
for the proof of \refT{rem:max-clique} it thus suffices to verify that the proof of \refT{thm:max-clique} from \refS{sec:main} can indeed be made to work (in the case~$m=1$, by suitable minor modifications) with the new parameter choice 
\begin{equation}\label{def:delta:m1:asymp}
\delta := \underbrace{\min\biggcpar{1-5\rho/2, \ \rho}}_{ = 1-5\rho/2 = 5\eps/2} -\frac{9 \log(\log(n))}{\log(n)} ,
\end{equation}
where for evaluating the minimum we used that~$\rho=\log_n(1/p)=2/5-\eps$ satisfies~$1/3 \le \rho \le 2/5$. 

\subsection{Application of old choice~\eqref{eq:m1:delta} of~$\delta$}\label{sec:lem:density:copy}
Recall that in~\refS{sec:main} we previously used~$\delta = \min \{\eps,\sigma/2\}$, see~\eqref{eq:m1:delta}. 
In the proof of \refT{thm:max-clique}, 
the bound~$\delta \le \sigma/2$ is only used in the deferred proof of \refL{lem:density} in \refS{sec:lem:density}, 
more precisely in~\eqref{eq:density:3} to establish~that
\begin{equation}\label{eq:density:3:goal}
n^{\delta}p \ll \log^{-7}(n).
\end{equation}
We now show that~\eqref{eq:density:3:goal} remains valid for our new choice~\eqref{def:delta:m1:asymp} of~$\delta$, 
which is routine: using~$p = n^{-\rho}$ and~$\delta \le \rho - 9 \log(\log(n))/\log(n)$ we here readily infer~that
\[
n^{\delta}p = n^{\delta-\rho} \le \log^{-9}(n) \ll \log^{-7}(n) .
\]

In the proof of \refT{thm:max-clique}, the other bound~$\delta \le \eps$ from~\eqref{eq:m1:delta} is only used in three places.
First, in~\refS{sec:main} below \refL{lem:choices} it is used in the arguments leading to $\min\{\eps, 1/200\} \le \delta \le 1/2$: 
in the conclusion of \refT{rem:max-clique} this estimate is replaced by $\delta = (1+o(1))5\eps/2$, 
which in turn readily follows from~$\eps \gg \log(\log(n))/\log(n)$ and the new definition~\eqref{def:delta:m1:asymp} of~$\delta$.
Second, in~\refS{sec:ass:s} the aforementioned old inequality $\delta \ge \min\{\eps, 1/200\} = \Omega(1)$ was also used in the proof of inequality~\eqref{eq:ass:s}, 
which for our new choice~\eqref{def:delta:m1:asymp} of~$\delta= (1+o(1))5\eps/2 \gg \log(\log(n))/\log(n)$ carries over by noting that here~$(s+1)/m = s+1 \ge (1-o(1)) \delta \log(n)/p \gg 1$, with room to spare. 
Third, in \refS{sec:technical:delta} the bound~$\delta \le \eps$ is also used to establish the upper bound~\eqref{ass:Janson-delta-bound} on~$\delta$,
whose applications are the subject of the next subsection (as we shall see, they will again remain valid). 

\subsection{Application of old inequality~\eqref{ass:Janson-delta-bound} for~$\delta$}\label{sec:main:rem:janson}
In the proof of \refT{thm:max-clique} from~\refS{sec:main}, inequality~\eqref{ass:Janson-delta-bound} for~$\delta$ is only used in two places.
Firstly, inequality~\eqref{ass:Janson-delta-bound} is used in the proof of \refL{lem:Janson} in \refS{sec:janson}, 
more precisely in~\eqref{eq:Janson:exponent:1}--\eqref{eq:Janson:exponent:2} to show that the
exponent~\eqref{eq:XW:bound} of Janson's inequality satisfies
\begin{equation}\label{eq:Janson:goal}
\frac{\mu^2}{2(\mu+\Delta)} \gg |W| \log n .
\end{equation}
Secondly, inequality~\eqref{ass:Janson-delta-bound} is used in the proof of~\eqref{eq:l1W:lower_w_choice} in \refS{sec:aux:ell0_np}, 
to establish the technical~estimate 
\begin{equation}\label{eq:l1W:lower_w_choice:goal}
\min\Bigcpar{\ell_0, \: n^{1-\delta}p} \gg n^{1/2}.
\end{equation}

We thus need to show that the two estimates~\eqref{eq:Janson:goal}--\eqref{eq:l1W:lower_w_choice:goal} remain valid for our new choice~\eqref{def:delta:m1:asymp} of~$\delta$, 
and we start with the routine proof of~\eqref{eq:l1W:lower_w_choice:goal}: 
by proceeding as in~\eqref{eq:l1W:lower} and~\eqref{eq:l1W:lower_w_choice}, 
using~$\delta \leq 1 - 5\rho/2 - 9 \log(\log(n))/\log(n)$ and~$\rho =\log_n(1/p) \ge 1/3$ it here readily follows that 
\begin{equation}\label{eq:l1W:lower_w_choice:copy}
\min\Bigcpar{\ell_0, \: n^{1-\delta}p} \gg n^{1-\delta}p \log^{-2}(n) = \frac{n^{1-\delta-\rho}}{\log^2(n)} \ge n^{3 \rho/2} \log^7(n) \gg n^{1/2}.
\end{equation}

We now turn to the key estimate~\eqref{eq:Janson:goal}, which is the crux of the matter.  
Recalling that~$m=1$, \refR{rem:Janson} implies that the exponent in the right-hand side of Janson's inequality~\eqref{eq:XW:bound} can here be written as 
\begin{equation}\label{eq:Janson:exponent:m1:copy}
\frac{\mu^2}{2(\mu+\Delta)} 
\ge 
\frac{\nu}{4} \cdot \min\left\{\frac{(a/k^2)\cdot (b/k^2)p}{k^2} , \; \frac{(a/k^2) \cdot (b/k^2)^{k-1} p^{\binom{k}{2}}}{k^2}\right\}.
\end{equation}
Recall that~$b \ge \ell_1(W)/(4m)$. 
Using~$a \ge |W|/4$ as well as~$m \le k=3$ and $\delta \leq 1 - 2\rho - 9 \log(\log(n))/\log(n)$, 
note that the arguments leading to estimate~\eqref{eq:Janson:exponent:1} here imply (with room to spare)~that 
\begin{align}\label{eq:Janson:exponent:1:m1}
\frac{(a/k^2) \cdot (b/k^2)p}{k^2|W| \log n}
\geq \frac{\ell_1(W)p}{4 k^6 m \log n} \cdot  \frac{a}{|W|} \gg \frac{n^{1-\delta}p^2}{\log^{4}(n)} 
= \frac{n^{1-\delta -2\rho}}{\log^{4}(n)} \gg 1.
\end{align}
Similarly, using~$a \ge |W|/4$, $b \ge \ell_1(W)/(4m)$ as well as~$m \le k=3$ and~$\delta \le 1-\rho(k/2+1)- 9 \log(\log(n))/\log(n)$, 
note that the arguments leading to estimate~\eqref{eq:Janson:exponent:2} here imply (with room to spare)~that 
\begin{equation}\label{eq:Janson:exponent:2:m1}
\begin{split}
\frac{(a/k^2) \cdot (b/k^2)^{k-1} p^{\binom{k}{2}}}{k^2|W| \log n} & \ge \frac{1}{k^4 \log n} \cdot \frac{a}{|W|} \cdot \biggpar{ \frac{\ell_1(W)p^{k/2}}{4mk^2}}^{k-1}\\
& \gg \frac{1}{\log^{2}(n)} \cdot \biggpar{ \frac{n^{1-\delta}p^{k/2+1}}{\log^{2}(n)}}^{k-1}\\
& \ge \biggpar{ \frac{n^{1-\delta-\rho(k/2+1)}}{\log^{4}(n)}}^{k-1} \gg 1 .
\end{split}
\end{equation}
Combining~\eqref{eq:Janson:exponent:1:m1}--\eqref{eq:Janson:exponent:2:m1} with~\eqref{eq:Janson:exponent:m1:copy} and~$\nu = \Omega(1)$ then establishes the desired estimate~\eqref{eq:Janson:goal}, as claimed.

\subsection{Completing the proof of \refT{rem:max-clique}}
With the discussed minor modifications from Sections~\ref{sec:lem:density:copy} and~\ref{sec:main:rem:janson} in hand, 
now all remaining arguments in the proof of \refT{thm:max-clique} carry over unchanged (in the relevant case~${m=1}$). 
This completes the proof of \refT{rem:max-clique}. 
\qed

\section{Deferred standard proofs: edge-density arguments}\label{sec:deferred}
In this section we give the deferred routine proofs of Lemmas~\ref{lem:non-nbr:lb} and~\ref{lem:density} from Sections~\ref{sec:mainproof} and~\ref{sec:main}.
Both proofs are based on edge-density arguments, 
i.e., use Chernoff bounds as well as union bound and counting~arguments.

\subsection{Non-neighbors and neighbors: proof of \refL{lem:non-nbr:lb}}\label{sec:lem:non-nbr:lb}
\begin{proof}[Proof of \refL{lem:non-nbr:lb}]
We start with the first assertion about the number of mutual non-neighbors. 
Given a set~$S \subseteq V$ of size~$|S| \le s$, we denote by~$X_S$ the number of its mutual non-neighbors in~$\Gnp$.
Note that~$X_S$ has distribution $\Bin\bigpar{n-|S|,(1-p)^{|S|}}$. 
Since~$|S| \le s \le \delta\log_{1/(1-p)}(n)$ by definition~\eqref{def:s} of~$s$, it follows that 
\begin{equation}\label{eq:lem:non-nbr:lb:exp}
\E X_S = (n - |S|) \cdot (1-p)^{|S|} 
\geq (n-s)n^{-\delta} = n^{1-\delta}(1-s/n) .
\end{equation}
In view of~$s \le \delta \log(n)/[-\log(1-p)] \le \log(n)/p$, with foresight we~define
\[
\tau := \max\Biggcpar{\frac{2 \log n}{np}, \: \sqrt{\frac{32 \log^2(n)}{n^{1-\delta}p[1-\log(n)/(np)]}}} \ge \max\biggcpar{\frac{2s}{n}, \: \sqrt{\frac{32 s\log n}{n^{1-\delta}(1-s/n)}}} .
\]
Using the assumption~$n^{1-\delta}p \gg \log^2(n)$, it is routine to see that~$np  \ge n^{1-\delta}p \gg \log n$ as well as~$\tau=o(1)$ and~$s/n = o(1)$.   
Using a union bound argument and standard Chernoff bounds (such as~\cite[Theorem~2.1]{JLR}), 
the probability that we have~${X_S \le (1-\tau/2)\E X_S}$ for some set~$S \subseteq V$ of size~$|S| \le s$ is at~most
\begin{equation}\label{eq:lem:non-nbr:lb:chern}
\sum_{S \subseteq V: |S| \le s} \exp\biggpar{-\frac{(\tau/2)^2 \E X_S}{2}} \le n^{s+1} \cdot e^{- (1-o(1))\tau^2 n^{1-\delta}/8} \le n^{-2s} \ll 1 .
\end{equation}
Noting that~$(1-\tau/2)\E X_S \ge (1-\tau)n^{1-\delta}$ by~\eqref{eq:lem:non-nbr:lb:exp} and choice of~$\tau$, 
then establishes the first assertion of~\refL{lem:non-nbr:lb}. 

We now turn to the second assertion about the number of neighbors. 
Given a vertex~$v \in V$, we denote by~$X_v$ the number of its neighbors in~$\Gnp$.
Note that~$X_v$ has distribution~$\Bin(n-1,p)$, with~$\E X_v \le n \cdot p$. 
Using a union bound argument and standard Chernoff bounds (such as~\cite[Theorem~2.1]{JLR}), 
in view of~$np \ge n^{1-\delta}p \gg \log^2(n)$ we see that 
the probability that we have~${X_v \ge 2np}$ for some vertex~$v \in V$ is at~most, say, 
\begin{equation*}
\sum_{v \in V}\exp\biggpar{-\frac{np}{4}} \le n \cdot o(n^{-2}) \ll 1. 
\end{equation*}
This establishes the second assertion of~\refL{lem:non-nbr:lb}, completing the proof of \refL{lem:non-nbr:lb}. 
\end{proof}

\subsection{Pseudorandom properties: proof of \refL{lem:density}}\label{sec:lem:density}
\begin{proof}[Proof of \refL{lem:density}]
Let~$\cU$ denote the set of all~$U \subseteq V$ of size~$|U| \ge \ell_0/\log^2(n)$. 
We decompose 
\[\cD_{U,W} = \cD_{U,W,\ref{lem:density:i}} \cap \cD_{U,W,\ref{lem:density:ii}} \cap \cD_{U,W,\ref{lem:density:iii}},
\]
corresponding to the three properties of~$\cD_{U,W}$.
The crux is that if~$\cD_{U,W}$ fails for some set~$W \subseteq V$ and~$U \subseteq W$, then by edge-monotonicity also the event
\[
\cD_{U} := \cD_{U,\ref{lem:density:i}} \cap \cD_{U,\ref{lem:density:ii}} \cap \cD_{U,\ref{lem:density:iii}}
\]
fails,
where the definition of each~$\cD_{U,x}$ is the same as~$\cD_{U,W,x}$ except that the restriction `vertices in~$V \setminus W$' is replaced by `vertices in~$V \setminus U$'. 
It thus suffices to show that the probability that~$\cD_{U}$ fails for some~$U \in \cU$ has probability~$o(1)$, 
and in the following we shall deal with each property~$\cD_{U,x}$ separately. 

We first focus on~$\cD_{U,\ref{lem:density:i}}$.  
Note that if~$\cD_{U,\ref{lem:density:i}}$ fails, then for some integer~$i \ge 1$ with~$r_i p \le 1$ there is a set of vertices~${X \subseteq V \setminus U}$ of size~$|X|=\ceil{x_i}$ such that there are at least $\ceil{x_i} \cdot \ceil{r_i p |U|} \ge r_i \ceil{x_i} |U|p$ edges between~$X$ and~$U$. 
Since the number of these edges has distribution~$\Bin(\ceil{x_i}|U|,p)$, 
using Chernoff bounds for the upper tail (see~\cite[Theorem~2.1]{JLR}) and a standard union bound argument it follows~that 
\begin{equation}\label{eq:chernoff:1}
\Pr(\neg \cD_{U,\ref{lem:density:i}}) \le \sum_{i \ge 1: r_i p \le 1}\binom{n}{\ceil{x_i}} \cdot e^{-\phi(r_i-1) \ceil{x_i} |U|p} ,
\end{equation}
where $\phi(x) = (1+x)\log(1+x) - x$ is defined as above~\eqref{eq:param}. 
Taking all~$U \in \cU$ into account, 
using first assumption~\eqref{eq:density:1}, 
and afterwards that $\ell_0/\log^2(n) \ge \log^2(n)$ by~\eqref{ass:partition}
together with the observation that $r_ip \leq 1$ implies $i \le \log^{O(1)}(n)$ by~\eqref{eq:param}, 
it follows via a standard union bound argument~that 
\begin{equation}\label{eq:chernoff:1:UB}
\begin{split}
\Pr(\neg \cD_{U,\ref{lem:density:i}} \text{ for some~$U \in \cU$}) & \leq \sum_{\ell_0/\log^2(n) \le u \le n}\binom{n}{u} \cdot \sum_{i \ge 1: r_i p \le 1} n^{\ceil{x_i}}e^{-\phi(r_i-1) \ceil{x_i}up}\\ 
& \leq \sum_{\ell_0/\log^2(n) \le u \le n} \: \sum_{i \ge 1: r_i p \le 1}\biggsqpar{ \frac{ne}{u} \cdot n^{\ceil{x_i}/u} \cdot e^{-\phi(r_i-1) \ceil{x_i}p}}^{u}\\
& \leq \sum_{\ell_0/\log^2(n) \le u \le n} \:  \sum_{i \ge 1: r_i p \le 1} \biggsqpar{ \frac{n\log^2(n)e}{\ell_0} \cdot e^{-\ceil{x_i}\bigpar{\phi(r_i-1) p-\log^3(n)/\ell_0}}}^{u}\\
& \leq \sum_{\ell_0/\log^2(n) \le u \le n} \:  \sum_{i \ge 1: r_i p \le 1} e^{-u}
\leq n \cdot \log^{O(1)}(n) \cdot e^{-\log^2(n)} \ll 1.
\end{split}
\end{equation}

Next, the argument for~$\cD_{U,\ref{lem:density:ii}}$ is similar (but simpler) than for~$\cD_{U,\ref{lem:density:i}}$.
Indeed, replacing~$\ceil{x_i}$ and~$r_ip$ with~$m+1$ and~$\alpha$ in the Chernoff bound, 
using assumptions~\eqref{eq:density:2} and~\eqref{ass:partition} 
it follows similarly to~\eqref{eq:chernoff:1}--\eqref{eq:chernoff:1:UB} that
\begin{equation*}
\begin{split}
\Pr(\neg \cD_{U,\ref{lem:density:ii}} \text{ for some~$U \in \cU$}) & \leq \sum_{\ell_0/\log^2(n) \le u \le n} \binom{n}{u}\binom{n}{m+1} \cdot e^{-\phi(\alpha/p - 1) (m+1)up}\\
& \leq \sum_{\ell_0/\log^2(n) \le u \le n} \biggsqpar{ \frac{n\log^2(n)e}{\ell_0} \cdot e^{-(m+1)\bigpar{\phi(\alpha/p - 1)p-\log^3(n)/\ell_0}}}^{u} \\
& \leq \sum_{\ell_0/\log^2(n) \le u \le n} e^{-u} \leq n \cdot e^{-\log^2(n)} \ll 1.
\end{split}
\end{equation*}

Finally we turn to~$\cD_{U,\ref{lem:density:iii}}$, for which we shall use a completely different approach (using counting arguments and contradiction). 
Given a vertex~$v \in V$, we denote by~$\Gamma(v)$ the neighborhood of~$v$ in~$G_{n,p}$. 
Let~$\cN$ denote the event that~$X_{v,w}:=|\Gamma(v) \cap \Gamma(w)| \le 2np^2$ for all distinct vertices~$v,w \in V$. 
Note that~$X_{v,w}$ has distribution~$\Bin\bigpar{n-2,p^2}$, with~$\E X_{v,w} \le np^2$. 
Using a union bound argument and standard Chernoff bounds (such as~\cite[Theorem~2.1]{JLR}), 
in view of the lower bound~$np^2 \gg n^{1-4/5} \gg \log(n)$ it routinely follows that
\begin{equation*}
\begin{split}
\Pr(\neg \cN) \leq n^2 \cdot \exp\biggpar{-\frac{np^2}{4}}
\ll 1 .
\end{split}
\end{equation*}
Since~$\cN$ holds whp, it remains to deterministically argue that~$\cN$ implies the event~$\cD_{U,\ref{lem:density:iii}}$ for any~${1 \le j \le 9 \log(n)}$ and set~${U \subseteq V}$ of size~${|U| \ge \ell_0/\log^2(n)}$. 
Aiming at a contradiction, suppose that there are~$j+1$ distinct vertices~$v_1,v_2,\ldots,v_{j+1} \in V\setminus U$ 
that are each adjacent to at least~${\bigpar{1/(j+1) + 1/\log^3(n)} \cdot |U|}$ vertices of~$U$. 
With foresight, note that assumptions~\eqref{eq:p:small}--\eqref{eq:m1:delta} and~$\sigma = \Omega(1)$ imply (with room to spare) the upper bound
\begin{equation}\label{eq:density:3}
n^{\delta}p \le n^{\delta-\sigma} \le n^{-\sigma/2} \ll \log^{-7}(n). 
\end{equation} 
Since~$s \le \log (n)/p$ as in the proof of~\refL{lem:non-nbr:lb}, in view of the definition~\eqref{def:ell1} of~$\ell_1{=\ell_0(\emptyset)}$ we see that~$\ell_0 \ge (1-\tau)n^{1-\delta}p/\log(n)$. 
Using the codegree-property~$\cN$ and the lower bound~${|U| \ge \ell_0/\log^2(n)}$ together with~$\tau =o(1)$ 
and the upper bound~$n^{\delta}p \ll \log^{-7}(n)$ from~\eqref{eq:density:3},  
for any~${1 \le i \le j+1}$ it follows~that
\[
\frac{\sum_{1 \le \ell \le j+1: i \neq \ell}|\Gamma(v_i) \cap \Gamma(v_\ell)|}{|U|} \le \frac{j \cdot 2np^2}{\Omega\bigpar{n^{1-\delta}p/\log^3(n)}} \le O\bigpar{n^{\delta}p \log^4(n)} \ll 1/\log^3(n) . 
\]
Since~$|U \cap \Gamma(v_i)| \ge {|U|/(j+1)+|U|/\log^3(n)}$, we thus obtain a contradiction by noting~that 
\[
|U| \; \ge \; \Bigl|\bigcup_{1 \le i \le j+1} \bigpar{U \cap \Gamma(v_i)}\Bigr| \ge \sum_{1 \le i \le j+1}\Bigpar{|U \cap \Gamma(v_i)|- \sum_{1 \le \ell \le j+1: i \neq \ell}|\Gamma(v_i) \cap \Gamma(v_\ell)|} >  |U| ,
\]
which completes the proof of \refL{lem:density}, as discussed. 
\end{proof}

\section{Upper bounds on clique chromatic number}\label{sec:upper}
In this section we prove new upper bounds on the clique chromatic number~$\chi_c(\Gnp)$ of~$\Gnp$ for~$n^{-2/5} \ll p \ll 1$, 
which are used in the proofs of our main results \refT{thm:main1}--\ref{thm:main3}. 
Our first upper bound demonstrates that the typical asymptotics $\chi_c(\Gnp) = \bigpar{\tfrac{1}{2} + o(1)}  \log(n)/p$ for $n^{-o(1)} \le p \ll 1$ (see \refT{thm:main1} and \refR{rem:upper-bound}) do not extend to smaller edge-probabilities~$p=p(n)$. 
Indeed, \refL{lem:upper-bound} below shows that the leading constant must be smaller than~$1/2$ for~$n^{-2/5} \le p \le n^{-\Omega(1)}$, since then~$\rho =\log_n(1/p) = \Omega(1)$ in~\eqref{eq:upper:sparse} below. 
%
\begin{lemma}\label{lem:upper-bound}
Suppose that~$\omega=\omega(n) \gg 1$. 
If the edge-probability $p=p(n)$ satisfies~$n^{-1/2} \ll p \le \log^{-\omega}(n)$, then~whp 
\begin{equation}\label{eq:upper:sparse}
    \chi_c\bigpar{\Gnp} \: \le \:     \Bigpar{\tfrac{1}{2} - \rho\bigpar{\tfrac{1}{2}-\rho} + o(\rho)}  \log(n)/p.
\end{equation}
\end{lemma}
\begin{remark}\label{rem:upper-bound} 
The proof 
implies that whp $\chi_c(\Gnp) \le \bigpar{\tfrac{1}{2} + o(1)}  \log(n)/p$ when~$n^{-1/2} \ll p \ll 1$.  
\end{remark}
Our second upper bound demonstrates that $\chi_c(\Gnp)=o(\log(n)/p)$ for~$p=n^{-2/5+o(1)}$, 
which is surprising in view of previous bounds and speculations~\cite{AK2018,MMP2019,LMW2023,DZ2023} for~$p \gg n^{-2/5}  (\log n)^{3/5}$. 
In fact, the upper bound~\eqref{eq:upper:sparse-1} of  \refL{lem:upper-bound-epsilon-1} is best possible for $(\log n)^{\omega}n^{-2/5} \leq p \ll n^{-1/3}$, see the proof of \refT{thm:main3} in~\refS{sec:mainproof}. 
\begin{lemma}\label{lem:upper-bound-epsilon-1}
If the edge-probability $p=p(n)$ satisfies~$p = n^{-2/5+\eps}$ with $2 n^{-2/5} \le p \le n^{-1/9}$, then~whp 
\begin{equation}\label{eq:upper:sparse-1}
    \chi_c\bigpar{\Gnp} \: \le \:     
    \bigpar{\tfrac{5}{2} + o(1)} \eps\log(n)/p.
\end{equation}
\end{lemma}

The proofs of \refL{lem:upper-bound} and~\ref{lem:upper-bound-epsilon-1} are given in the next two subsections, and they both follow similar
two-step approaches to construct a valid vertex coloring of~$\Gnp$ without monochromatic inclusion-maximal cliques. 
Aiming at roughly~$\chi_c(\Gnp) \le \delta \log(n)/p$ in~\eqref{eq:upper:sparse} and~\eqref{eq:upper:sparse-1} for suitable~$\delta>0$, 
in the first step we use a simple greedy approach to color the vertices of~$\Gnp$ using~$\delta \log (n)/p$ colors, until a set~$N$ of $|N|=\Theta(n^{1-\delta})$ uncolored vertices is left.
In the second step we then use a lemma-specific argument to color the remaining vertices in~$N$ using a negligible number of additional new colors 
(compared to the number of colors used in the first~step). 

\subsection{Leading constant less than~$1/2$: proof of \refL{lem:upper-bound}}\label{sec:upper-bound}
In the below proof of \refL{lem:upper-bound}, in the second coloring step we 
partition the set~$N=N_1 \cup \cdots \cup N_z$ into~$z=\Theta(1/p)$ parts, 
and then show that whp there are no cliques in the induced subgraphs~$\Gnp[N_i]$ that are inclusion maximal with respect to~$\Gnp$. 
By giving all vertices in each~$N_i$ one new color, it turns out that we obtain a valid coloring of the remaining vertices in~$N$ using~$z=o(\rho \log(n)/p)$ many additional new~colors. 
\begin{proof}[Proof of \refL{lem:upper-bound} and \refR{rem:upper-bound}]
Recalling that~$\rho =\log_n(1/p)$, with foresight we define
\[
s := \Bigfloor{\delta \log_{\tfrac{1}{1-p}}(n)}, 
\qquad
\delta := \frac{1}{2} - \frac{\rho}{2} + \frac{\rho^2}{1-\lambda} + \lambda, 
\qquad 
\lambda := \frac{6\log \log(n)}{\log(n)}
\quad \text{ and } \quad
z := \biggceil{\frac{4}{p}} .
\]
To complete the proof, we claim that it suffices to show that if~${n^{-1/2} \ll p \ll 1}$, then whp 
\begin{equation}\label{eq:upper:sparse:refined}
\chi_c(\Gnp) \le {s+z+1} .
\end{equation}
Indeed, this bound implies estimate~\eqref{eq:upper:sparse}, since~$p \le \log^{-\omega}(n)$ ensures $\log_{1/(1-p)}(n) = (1+O(p))\log(n)/p$ and~$\max\{p,\lambda\} \ll \log(1/p)/\log(n) = \rho$ as well as~$z \ll \log(1/p)/p = \rho \log(n)/p$.
This bound also implies the estimate of \refR{rem:upper-bound}, since~$\log^{-\omega}(n) \leq p \ll 1$ ensures~$\rho,\lambda = o(1)$ and $\log_{1/(1-p)}(n) = (1+o(1))\log(n)/p$.

We now turn to the proof of~\eqref{eq:upper:sparse:refined}.   
Our approach is to color the vertices of~$\Gnp$ with the colors~$\{1, \ldots, s+z+1\}$ using the following procedure, 
where we fix an ordering~$v_1, \ldots, v_n$ of all vertices (say using lexicographic ordering). 
First we sequentially consider the vertices $v_1, \ldots, v_s$, and each time color all so far uncolored vertices in~$N(v_i)$ with color~$i$.
Then we color all so far uncolored vertices in~$\{v_1, \ldots, v_s\}$ with color~$s+1$, 
so that the set of all so far  uncolored vertices~equals
\[
N: = V \setminus \Bigpar{S \cup \bigcup_{v_i \in S} N(v_i) } \qquad \text{ with } \qquad S:=\{v_1, \ldots, v_s\} .
\]
Using the ordering of the vertices fixed above, we then partition~$N$ into disjoint sets 
\[
N = N_1 \cup \cdots \cup N_{z} \qquad \text{ with } \qquad |N_i| \le 2 |N|/z,
\]
and, for each~$1 \le i \le z$, then color all vertices in~$N_i$ with color~$s+1+i$.

Let us collect some basic properties of the resulting coloring.
First we record that, by a routine argument\footnote{We can reuse the setup from the proof of \refL{lem:non-nbr:lb} in \refS{sec:lem:non-nbr:lb} by noting that $|N| = X_S$ holds, where~$S \subseteq V$ has size~$|S|=s$.  
Using a standard union bound argument, we thus have 
$\Pr(|N| \ge 2 n^{1-\delta}) \le {\sum_{S \subseteq V: |S|=s}\Pr(X_S \ge 2 n^{1-\delta})}$. 
Since $n^{1-\delta}p \gg (\log n)^2$ implies~$\E X_S = n^{1-\delta}(1+o(1)) \gg \log(n)/p \ge s$ in~\eqref{eq:lem:non-nbr:lb:exp}, 
standard Chernoff bounds imply $\Pr(X_S \ge 2 n^{1-\delta}) \le \exp\bigpar{-\Theta(\E X_S)} \ll n^{-3s}$, say, 
which analogous to~\eqref{eq:lem:non-nbr:lb:chern} yields $\Pr(|N| \ge 2 n^{1-\delta}) \le n^s \cdot n^{-3s} = n^{-2s} \ll 1$, as~desired.} 
analogous to the non-neighbors argument around~\eqref{eq:lem:non-nbr:lb:exp}--\eqref{eq:lem:non-nbr:lb:chern} in \refS{sec:lem:non-nbr:lb} (exploiting that~$n^{1-\delta}p \gg \log^2(n)$ holds), 
we obtain that whp $|N| \le 2 n^{1-\delta}$. 
%
Next, note that every vertex colored~$i \le s$ is adjacent to vertex~$v_i$, which itself has a different color than~$i$ (either color~$s+1$ or a color that is less than~$i$).
Furthermore, the set of vertices colored~$s+1$ form an independent set.
From this we deduce the following: if~$\Gnp$ contains an inclusion-maximal monochromatic clique~$K$, then its color must be in~$\{s+2, \ldots, s+z+1\}$, 
which means that~$K$ must be contained in some set~$N_i$. 
To complete the proof, it thus suffices to show that whp no~$N_i$ contains an inclusion-maximal~clique. 

To this end we expose the status of potential edges of~$\Gnp$ in two rounds:
in the first round we expose the edge-status of all vertex pairs containing at least one vertex from~$S$ (which contains enough information to determine~$N$), 
and in the second round we expose the edge-status of all remaining vertex pairs.
We henceforth condition on the outcome of the first exposure round, and assume that~$|N| \le 2 n^{1-\delta}$ (since this holds whp). 
As usual, to avoid clutter, we shall omit this conditioning from out notation. 
Fix~$N_i$ with~$1 \le i \le z$. We infer~that
\[
|N_i| \le 2|N|/z \le n^{1-\delta} p.
\]
Note that if~$K \subseteq N_i$ is an inclusion-maximal clique in~$\Gnp$, then $K$ is also an inclusion-maximal clique in~${\Gnp[V \setminus S]}$. 
Denoting by~$X_{i,k}$ the number of inclusion-maximal cliques~$K \subseteq N_i$ of size~$|K|=k$ in~${\Gnp[V \setminus S]}$, 
writing~$X_i:={\sum_{2 \le k \le |N_i|}X_{i,k}}$ it therefore routinely follows~that  
\begin{equation}\label{eq:pr:Si}
\Pr(\text{$N_i$ contains an inclusion-maximal clique}) \le \Pr(X_i \ge 1) \le \sum_{2 \le k \le |N_i|} \E X_{i,k}.
\end{equation}
Note that after conditioning on the outcome of the first exposure round, all potential edges inside~$V \setminus S$ are still included independently with probability~$p$. 
Hence standard random graph estimates give, in view of~$|S|=s = o(n)$ and~$k \le |N_i| \le n^{1-\delta}p = n^{1-\delta-\rho/2} p^{1/2}  = o(n)$, that, say,  
\begin{align*}
\E X_{i,k} & = \binom{|N_i|}{k}p^{\binom{k}{2}}\bigpar{1-p^k}^{n-|S|-k} \le \Bigpar{en^{1-\delta}p/k \cdot  p^{(k-1)/2} \cdot e^{-(1-o(1))p^kn/k}}^k.
\end{align*}
We now proceed by several case distinctions.
First we assume that~$p^kn \ge 4 k \log(n)$ holds, where it follows~that 
\begin{equation}\label{eq:EXik:1}
\E X_{i,k} \le \Bigpar{ene^{-(1-o(1))p^kn/k}}^k \le n^{-k} .
\end{equation}
We henceforth assume that~$p^kn < 4 k \log(n)$ holds, where in view of~$p=n^{-\rho}$ we see~that 
\begin{align*}
\E X_{i,k} \le \Bigpar{en^{1-\delta}\cdot  p^{1/2 + k/2}/k}^k \le \Bigpar{e n^{1-\delta} p^{1/2} \sqrt{4\log(n)/(nk)}}^k \le \Bigpar{e^2 n^{1/2-\delta-\rho/2}\sqrt{\log(n)/k}}^k .
\end{align*}
In particular, 
if~$p^kn < 4 k \log(n)$ and~$k \ge e^{6} \log n$ both hold, 
then using~$\delta \ge 1/2-\rho/2$ we see~that 
\begin{equation}\label{eq:EXik:2}
\E X_{i,k} \le \Bigpar{e^{-1}n^{1/2-\delta-\rho/2}}^k \le e^{-k} .
\end{equation}
We now consider the remaining case when~$p^kn < 4 k \log(n)$ and~$k < e^{6} \log n$ both hold,
where we have $p^k n < 4 k \log(n) \ll \log^3(n)$ and thus $k \ge \log_{1/p}(n/\log^3(n)) \ge (1-\lambda)/\rho$, say. 
Using the form of~$\delta$ and~$n^{-\lambda} = \log^{-6}(n)$ together with~$k \ge (1-\lambda)/\rho$ and~$n^{-\rho}=p$, 
we then infer~that 
\begin{equation}\label{eq:EXik:3}
\E X_{i,k} \le \Bigpar{e^2 n^{1/2-\delta-\rho/2}\sqrt{\log(n)}}^k \le \Bigpar{n^{-\rho^2/(1-\lambda)-\lambda/2}}^k \le n^{-\rho} n^{-\lambda k /2} \le p/ \log^{k}(n) .
\end{equation}
To sum up, combining the estimates~\eqref{eq:EXik:1}, \eqref{eq:EXik:2} and~\eqref{eq:EXik:3} it follows that
\begin{align*}
\sum_{2 \le k \le |N_i|} \E X_{i,k} \le \sum_{k \ge 2}\Bigsqpar{n^{-k} + e^{-k}\indic{k \ge e^{6} \log n}  + p/\log^{k}(n)} \le o(n^{-1}) + o(p).  
\end{align*}
Combining this estimate with~\eqref{eq:pr:Si} and~$z = O(1/p) \ll n$ then shows that 
\[
\Pr(\text{some $N_i$ with~$1 \le i \le z$ contains an inclusion-maximal clique}) \le z \cdot \Bigsqpar{o(n^{-1}) + o(p)} = o(1) ,
\]
which completes the proof of Lemma~\ref{lem:upper-bound}, as discussed. 
\end{proof}

\subsection{Adaptive leading constant: proof of \refL{lem:upper-bound-epsilon-1}}\label{sec:lem:upper-bound-epsilon-1}
The below proof of \refL{lem:upper-bound-epsilon-1} uses a similar first coloring step as \refL{lem:upper-bound}, 
but in the second step we focus on the induced subgraph~$\Gnp[N]$ and argue as follows:
since~$\Gnp[N]$ has the same distribution as~$G_{|N|,p}$, by exploiting the non-monotone behavior of~$\chi_c(\Gnp)$ around~${p=n^{-2/5}}$ (see~\cite{MMP2019}) it turns out that whp~$\chi_c\bigpar{\Gnp[N]} = o(\eps \log(n)/p)$. 
Since inclusion-maximal cliques in~$\Gnp$ are also inclusion-maximal in~$\Gnp[N]$, 
we thus obtain a valid coloring of the remaining vertices in~$N$ using at most~$o(\eps \log(n)/p)$ many additional new~colors. 
\begin{proof}[Proof of \refL{lem:upper-bound-epsilon-1}]
With foresight we define
\[
s := \Bigfloor{\delta \log_{\tfrac{1}{1-p}}(n)}, 
\qquad
\delta := \frac{5\eps}{2}
\quad \text{ and } \quad
z := \biggceil{\frac{8}{p \sqrt{\log n}}} .
\]
Note that~$\eps \ge \log(2)/\log(n)$. Since~$p \le n^{-1/9}$ ensures $\log_{1/(1-p)}(n) = (1+O(p))\log(n)/p$ and~$p \ll \eps$ as well as~$z \ll \log(2)/p \le \eps \log (n)/p$, 
to complete the proof of~\eqref{eq:upper:sparse-1} it suffices to show that whp $\chi_c(\Gnp) \le {s+z+1}$. 

Similar to \refL{lem:upper-bound}, our approach is to color the vertices of~$\Gnp$ with the colors~$\{1, \ldots, s+z+1\}$ using the following procedure, 
where we fix an ordering~$v_1, \ldots, v_n$ of all vertices (say using lexicographic ordering). 
First we sequentially consider the vertices $v_1, \ldots, v_s$, and each time color all so far uncolored vertices in~$N(v_i)$ with color~$i$.
Then we color all so far uncolored vertices in~$\{v_1, \ldots, v_s\}$ with color~$s+1$, 
so that the set of all so far  uncolored vertices~equals
\[
N: = V \setminus \Bigpar{S \cup \bigcup_{v_i \in S} N(v_i) } \qquad \text{ with } \qquad S:=\{v_1, \ldots, v_s\} .
\]
Subsequently, our plan is to color the vertices in~$N$ with colors from the set $\{s+2,\cdots,s+z+1\}$, so that $G_{n,p}[N]$ contains no monochromatic inclusion-maximal clique. 
By an argument analogous to the one used in the proof of \refL{lem:upper-bound}, we deduce the following: if~$\Gnp$ contains a monochromatic inclusion-maximal clique~$K$, then~$K$ must be contained in the set~$N$. To complete the proof, it thus suffices to show that whp~$\chi_c\bigpar{\Gnp[N]} \le z$. 

To this end we expose the status of potential edges of~$\Gnp$ in two rounds:
in the first round we expose the edge-status of all vertex pairs containing at least one vertex from~$S$ (which contains enough information to determine~$N$), 
and in the second round we expose the edge-status of all remaining vertex pairs.
We henceforth condition on the outcome of the first exposure round, and assume that~$\tfrac{1}{2}n^{1-\delta} \le |N| \le 2 n^{1-\delta}$ (since this holds whp, by a routine argument analogous to  
the proof \refL{lem:non-nbr:lb}, as discussed in \refS{sec:upper-bound}). 
Note that after conditioning on the outcome of the first exposure round, all potential edges inside~$N$ are still included independently with probability~$p$, 
which implies that~$G_{n,p}[N]$ has the same distribution as~$G_{|N|,p}$. 
Since~$\delta=5\eps/2$, $p \le n^{-1/9}$ and $\tfrac{1}{2}n^{1-\delta} \le |N| \le 2 n^{1-\delta}$ imply~$p=n^{-2/5+\eps}=\Theta(|N|^{-2/5})$ and~$|N| = \Theta(1/p^{5/2}) \gg n^{1/4}$, by invoking Theorem~3.1 in \cite{MMP2019} (which for $p = \Theta(n^{-2/5})$ implies that whp $\chi_c(G_{n,p}) \le 2 n p^{3/2}/\sqrt{\log n}$) it then follows that, whp, 
\[
\chi_c\bigpar{G_{n,p}[N]} \leq \frac{2 \cdot |N| \cdot p^{3/2}}{\sqrt{\log |N|}} \leq \frac{2 \cdot 2n^{1-5\eps/2} \cdot n^{-3/5 + 3\eps/2}}{\sqrt{\tfrac{1}{4}\log n}} = \frac{8n^{2/5-\eps}}{\sqrt{\log n}} = \frac{8}{p \sqrt{\log n}} \le z,
\]
which completes the proof of \refL{lem:upper-bound-epsilon-1}, as discussed. 
\end{proof}

\section{Concluding remarks and open problems}\label{sec:concluding}
We close this paper with some remarks and open problems concerning the 
clique chromatic number~$\chi_c(\Gnp)$ of the binomial random graph~$\Gnp$ with edge-probability~$p=p(n)$. 

{\vspace{-0.125em}
\begin{itemize}[leftmargin=1.5em] 
\itemsep 0.5em \parskip 0.125em  \partopsep=0.125em \parsep 0.125em  
\item{\bf Order of magnitude.}  
The main remaining open problem for the clique chromatic number is to determine the typical order of magnitude of~$\chi_c(\Gnp)$ when~$p=p(n)$ is close to~$n^{-2/5}$ and~$n^{-1/3}$ 
(cf.~\refR{rem:np13:improvement}). 
\begin{problem}
Fix~$\eps>0$. Determine the whp order of magnitude of~$\chi_c(\Gnp)$ when~$n^{-2/5-\eps} \ll p \ll n^{-2/5+\eps}$ and~$\tfrac{1}{3}n^{-1/3} \le p \ll n^{-1/3+\eps}$. 
\end{problem}
These two ranges of~$p=p(n)$ are of particular interest, because around~${p=n^{-2/5}}$ and~${p=n^{-1/3}}$ there seems to be a phase transition in the structure of the valid colorings:
indeed, around these points the optimal lower bound strategies appear to change (in terms of which clique size matters), which makes the proof approaches from~\cite{MMP2019,LMW2023} run into technical difficulties. 
To further illustrate our limited understanding, we remark that a conjecture from~\cite[Section~5]{LMW2023} predicts that whp $\chi_c(\Gnp) = \Theta(\log(n)/p)$ for~$n^{-2/5}  (\log n)^{3/5} \ll p \ll 1$, 
whereas \refL{lem:upper-bound-epsilon-1} demonstrates that this conjecture is not always correct: 
indeed, \eqref{eq:upper:sparse-1} implies that whp $\chi_c(\Gnp) = o(\log(n)/p)$ for~$n^{-2/5} \ll p \le n^{-2/5+o(1)}$. 
It may be possible that the aforementioned conjecture from~\cite{LMW2023} is also incorrect around~$p=n^{-1/3}$. 
To stimulate more research into this intriguing possibility, we record that a simple modification of the proof of \refT{rem:max-clique} from \refS{sec:main:rem} yields the following lower bound. 
%
%
\begin{corollary}\label{lem:max-clique:3}
Suppose that~$\omega=\omega(n) \gg 1$. If~$p=p(n)$ satisfies~$(\log n)^{\omega}n^{-1/3} \leq p \ll n^{-1/3.75}$, then~whp
\begin{equation}\label{eq:maxclique:3}
    \chi_c\bigpar{\Gnp} \: \ge \:     \bigpar{3 +o(1) }\log\bigpar{n^{1/3}p}/p .
\end{equation}
\end{corollary}
It would be desirable to know if there is an upper bound on $\chi_c(\Gnp)$ which matches~\eqref{eq:maxclique:3}, 
since this could potentially lead to a variant of \refT{thm:main3} for a range of suitable of edge-probabilities~$p \ge (\log n)^{\omega}n^{-1/3}$. 
\item{\bf Asymptotics.} 
Another major open problem for the clique chromatic number is to determine the typical asymptotics of~$\chi_c(\Gnp)$ when~$p=p(n)$ is between~$n^{-1/3}$ and~$n^{-o(1)}$.
\begin{problem}
Fix~$\eps>0$. Determine the whp asymptotics of~$\chi_c(\Gnp)$ when~$n^{-1/3+\eps} \ll p \ll n^{-\eps}$. 
\end{problem}
In this range of~$p=p(n)$ we know that whp $\chi_c(\Gnp) = {\Theta(\log(n)/p)}$ by \refT{thm:main1}, so the main difficulty is to determine the behavior of the leading constant. 
This is of particular interest, since 
(i)~by \refL{lem:upper-bound} the leading constant must be smaller than~$1/2$, which contrasts 
the typical asymptotics $\chi_c(\Gnp) = {\bigpar{\tfrac{1}{2} + o(1)}  \log(n)/p}$ of \refT{thm:main2} for $n^{-o(1)} \le p \ll 1$, 
and 
(ii)~it may be possible that `adaptive' behavior akin to~\eqref{eq:main3} and~\eqref{eq:maxclique:3} plays a role in some range of~$p=p(n)$. 
The whp asymptotics of~$\chi_c(\Gnp)$ for constant~${p \in (0,1)}$ is a related problem of interest. 
A close inspection of its proof reveals that \refT{thm:main2} carries over to~$p \le p_0$ for some small constant~$p_0>0$ (via simple routine modifications), 
provided one uses~$\log_{\frac{1}{1-p}}(n)$ instead of~$\log (n)/p$ in the resulting estimate~\eqref{eq:main2}; 
we now record this observation for future~reference. 
\begin{corollary}\label{thm:main4}%
There exists~$p_0>0$ such that the following holds: if~$n^{-o(1)} \le p \le p_0$, then whp 
\begin{equation}\label{eq:main4}
    \chi_c\bigpar{\Gnp} \: = \:     \bigpar{\tfrac{1}{2}+o(1)}  \log_{\frac{1}{1-p}}(n).
\end{equation}
\end{corollary}
By the results of Demidovich and Zhukovskii~\cite{DZ2023} we know that \eqref{eq:main4} also holds whp for constant~$p \in [0.5,1)$.
We believe that it should be possible to close the gap for constant~$p \in (p_0,0.5)$ by adding more bells and whistles to our proof approach, 
which we intend to elaborate on in future work.  
%
\end{itemize}

\bigskip{\noindent\bf Acknowledgements.} 
We are grateful to the referees for a very careful reading of the paper, and for useful suggestions concerning the presentation. 
Lutz Warnke would also like to thank Lyuben Lichev and Dieter Mitsche for helpful discussions on the clique chromatic number of random graphs.

\bibliographystyle{plain}

\end{document}